\documentclass[UTF-8,reqno]{amsart}
\usepackage{enumerate}
\usepackage{mhequ}
\usepackage{amssymb,url,color, booktabs,nccmath}
\usepackage[left=3.2cm,right=3.2cm,top=4cm,bottom=4cm]{geometry}
\usepackage{mathrsfs}
\usepackage{enumitem,dsfont}

\usepackage{color}
\usepackage[colorlinks=true]{hyperref}
\hypersetup{
    linkcolor=blue,          
    citecolor=red,        
    filecolor=blue,      
    urlcolor=cyan
}

\definecolor{darkergreen}{rgb}{0.0, 0.5, 0.0}

\newtheorem{theorem}{Theorem}[section]
\newtheorem{lemma}[theorem]{Lemma}
\newtheorem{proposition}[theorem]{Proposition}

\newtheorem{corollary}[theorem]{Corollary}

\newtheorem{definition}[theorem]{Definition}
\theoremstyle{definition}
\newtheorem{remark}[theorem]{Remark}

\allowdisplaybreaks

\usepackage{esint}
\usepackage{cite}
\usepackage{verbatim}
\usepackage{times}
\usepackage{bm}

\numberwithin{equation}{section}

\begin{document}

\title{Sharp non-uniqueness of solutions to stochastic Navier-Stokes equations}
\author{Weiquan Chen}
\address[W. Chen]{Academy of Mathematics and Systems Science,
Chinese Academy of Sciences, Beijing 100190, China}
\email{chenweiquan@amss.ac.cn}

\author{Zhao Dong}
\address[Z. Dong]{Academy of Mathematics and Systems Science,
Chinese Academy of Sciences, Beijing 100190, China}
\email{dzhao@amt.ac.cn}

\author{Xiangchan Zhu}
\address[X. Zhu]{Academy of Mathematics and Systems Science,
Chinese Academy of Sciences, Beijing 100190, China}
\email{zhuxiangchan@126.com}
\thanks{Research  supported   by National Key R\&D Program of China (No. 2020YFA0712700) and the NSFC (No. {11931004,} 12090014, 12288201) and
  the support by key Lab of Random Complex Structures and Data Science,
 Youth Innovation Promotion Association (2020003), Chinese Academy of Science. The financial support by the DFG through the CRC 1283 ``Taming uncertainty and profiting
 from randomness and low regularity in analysis, stochastics and their applications'' is greatly acknowledged.}

\begin{abstract} In this paper we establish a sharp non-uniqueness result for stochastic $d$-dimensional ($d\geq2$) incompressible Navier-Stokes equations. First, for every divergence free initial condition in $L^2$ we show existence of infinite many global in time probabilistically strong and analytically weak solutions in the class $L^\alpha\big(\Omega,L^p_tL^\infty\big)$ for any $1\leq p<2,\alpha\geq1$. Second, we prove the above result is sharp in the sense that pathwise uniqueness holds in the class of $L^p_tL^q$ for some $p\in[2,\infty],q\in(2,\infty]$ such that $\frac2{p}+\frac{d}{q}\leq1$, which is a stochastic version of Ladyzhenskaya-Prodi-Serrin criteria. Moreover, for stochastic $d$-dimensional incompressible Euler equation, existence of infinitely many  global in time probabilistically strong and analytically weak solutions is obtained. Compared to the stopping time argument used in \cite{HZZ19, HZZ21a}, we developed a new stochastic version of the convex integration. More precisely, we introduce expectation during convex integration scheme and construct directly solutions on the whole time interval $[0,\infty)$.
\end{abstract}
\maketitle

\section{Introduction}\label{sec. introd.}
\setcounter{section}{1}\
The Navier-Stokes/Euler equations are fundamental models in fluid dynamics. Existence  of global strong solutions to the three dimensional incompressible Navier--Stokes system  is one of the Millennium Prize Problems. An intimately related question is that of uniqueness of solutions, which has been studied a lot in the literature.
In 2D case, existence and uniqueness of solution is well-known. In higher dimensions, since  existence of weak solutions is known {\cite{Ler34, Hop51},} there are a number of literature on the uniqueness of weak solutions. For brevity, we summarize the classical Ladyzhenskaya-Prodi-Serrin uniqueness criteria as follows:
\begin{theorem}[\cite{FJR72, Kat84, FLRT00, LM01, CL22}]\label{Thm. weak strong uniqueness} Let $d\geq2$ and $u$ be a weak solution to the incompressible Navier--Stokes equations such that $u$ belongs to
\begin{equation}\label{Serrin scaling spaces}
	X^{p,q}_T:=
	\left\{\begin{aligned}
		&~L^p\big(0,T;L^q\big)\ \ \ \ \ \ 1\leq p<\infty\\
		&~C\big([0,T];L^q\big)\ \ \ \ \ \ \ p=\infty\\
	\end{aligned}\right.
\end{equation}
for some $2\leq p\leq\infty$ and $d\leq q\leq\infty$ such that $2/p+d/q\leq1$. Then $u$ is unique in this class of weak solutions and is Leray-Hopf in the sense that $u\in  C_w\big([0,T];L^2\big)\cap L^2\big(0,T;H^1\big)$ and satisfies energy inequality.
\end{theorem}

In the literature the space $X^{p,q}$ is called sub-critical when $2/p+d/q<1$, critical when $2/p+d/q=1$, and super-critical when $2/p+d/q>1$. It is natural to ask what would happen in the supercritical regime $2/p+d/q>1$. i.e. whether uniqueness holds in this case. Recently Cheskidov and Luo \cite{CL21, CL22} proved the solutions are not unique in $L^p\big(0,T;L^\infty\big),1\leq p<2,d\geq2$ nor in $C\big([0,T],L^p\big),p<2,d=2$ by using the method of convex integration.
Convex integration was introduced into fluid dynamics by De Lellis and Sz{\'e}kelyhidi~Jr.
\cite{DLS09,DLS10, DLS13, DLS14}. This method has already led to a number of groundbreaking results: Isett \cite{Ise18} proved Onsager's conjecture, see also \cite{BDLSV19}. Non-uniqueness of weak solutions to the incompressible Navier-Stokes equations was obtained by Buckmaster and Vicol \cite{BV19}, see also Buckmaster, Colombo and Vicol \cite{BCV18}. Burczak, Modena and Sz\'{e}kelyhidi Jr. \cite{BMS21} then obtained ill-posedness for power-law flows and also, in particular, non-uniqueness of weak solutions to the Navier-Stokes equations for every given divergence free initial condition in $L^2$. We refer to the reviews \cite{BV19, BV21} for more details and references. We also mention  by a different method, a first non-uniqueness result for Leray solutions was obtained in \cite{ABC21} for the Navier-Stokes system with a special force.
\par In view of these negative developments, a suitable stochastic perturbation may provide a regularizing effect on problems. In the deterministic case, a selection of solutions depending continuously on the initial condition has not been obtained. However, the probabilistic counterpart, i.e. the Feller property and even the strong Feller property which corresponds to a smoothing with respect to the initial condition, ware established  by Da Prato and Debussche \cite{DPD03} and by Flandoli and Romito \cite{FR08}. A transport noise and linear multiplicative noise prevent blow up of strong solutions have been obtained by Flandoli, Gubinelli and Priola \cite{FGP10} and Flandoli and Luo \cite{FL21} and Glatt-Holtz and Vicol \cite{GHV14} and R\"ockner, Zhu and Zhu \cite{RZZ14}. One would naturally ask whether the  non-uniqueness result still holds in the stochastic case. Recently in \cite{HZZ19, HZZ21a, HZZ21b} non-uniqueness in law and even non-uniqueness of Markov family for the stochastic Navier-Stokes/Euler equations have been established by a stochastic counterpart of the convex integration method. One may further ask whether the noise makes the critical regularity of uniqueness different, i.e. whether non-uniqueness still holds in the supercritical regime.
\par In this paper we prove that sharp non-uniqueness holds for the stochastic $d$-dimensional ($d\geq2$) Navier-Stokes system driven by an additive noise. In \cite{HZZ19, HZZ21a, HZZ21b} the stopping time is introduced to control the noise uniformly in $\omega$ and it was removed by a suitable extension of solutions. However, such extensions require the solution at stopping time belongs to $L^2$-space which is not applicable in our case as the solution is only in $L^2$ space for a.e. $t$. Instead we introduce expectation during convex integration scheme which can be viewed as a new stochastic version of the convex integration. Since the nonlinear term is quadratic, we have to estimate higher moments at step $q$ than the moment bound we required at step $q+1$. Then it seems we have to bound all the finite moment at each step which may blow up during iterations. The key point that this method works is that the higher moments at step $q$ only depends on the parameters up the step $q$ and we could choose the parameters at the step $q+1$ to guarantee smallness. Moreover, to construct the solutions directly on the whole time interval $[0,\infty)$, we introduced the norm of the following form:
$$\sup_{s\geq0}\Big(\mathbb{E}\big\|u\big\|^\alpha_{L^p([s,s+1];L^\infty)}\Big)^{1/\alpha}$$
with $p,\alpha\geq1$. This requires the stochastic part also has finite norm of this form. To this end we introduced a damping term in the linear equation and subtract the extra term in the nonlinear equation (see \eqref{sto. low Re system} and \eqref{random NS} below for more details.)

\subsection{Main Results}

In this paper we are concerned with stochastic Navier-Stokes equations on $\mathbb{T}^d, d\geq2$ driven by an additive stochastic noise. The equations govern the time evolution of the fluid velocity $u$ and read as 
\\
\begin{equation}
	\left\{\begin{aligned}
		&{\rm d}u(t)=\Big(-{\rm div}\big(u(t)\otimes u(t)\big)+\Delta u(t)-\nabla p(t)\Big){\rm d}t+{\rm dW}(t)~,\\
		&{\rm div}\ u(t)=0~,\\
        &u(0)=u_0~.
	\end{aligned}\right.
	\label{sto. NS}
\end{equation}
\\
Here 
${\rm W}=\big\{{\rm W}(t);~0\leq t<\infty\big\}$ is a $GG^*$ Wiener process on a given probability space $\big(\Omega, \mathscr{F}, \mathbb{P}\big)$ and $G$ is a Hilbert-Schmidt operator from $L^2$ to $L^2$. 
Let $\lbrace\mathscr{F}_t\rbrace_{t\geq0}$ denote the  normal filtration generated by ${\rm W}$, that is, the canonical right-continuous filtraton augmented by all the $\mathbb{P}$-negligible events.

\par Compared to  the  deterministic case, the stochastic equations possess additional structural features. First of all, we distinguish between probabilistically strong and probabilistically weak (also called  martingale) solutions. Probabilistically strong solutions are constructed on a given probability space and are adapted with respect to the given  noise. Probabilistically weak solutions do not have this property: they are typically obtained by the method of compactness where the noise as well as the probability space becomes part of the construction.  For the Navier-Stokes equations only  probabilistically weak solutions are obtained by a compactness argument. In fact it is necessary to take expectation to control the noise and obtain uniform energy estimates, which then leads to probabilistically weak solutions. Due to the lack of uniqueness in higher dimensions we cannot apply Yamada--Watanabe's theorem to obtain probabilistically strong solutions. Moreover, if we analyze the equation $\omega$-wise, then the converging subsequence from compactness argument may depend on $\omega$ which destroys adaptedness. Consequently, it has been a long standing open problem to construct probabilistically strong solutions to the Navier--Stokes system \eqref{sto. NS} in higher dimensions, see page 84 in \cite{Fla08}. In \cite{HZZ21a}, Hofmanov\'{a}, Zhu and the third named author solved this problem and proved  existence of global-in-time non-unique probabilistically strong and analytically weak solutions for every given divergence free initial condition in $L^{2}$ in 3D case by using a stochastic convex integration method. Our main result also extends this result to higher dimensional case.
\par Now we recall the definition of probabisitcally strong and analytically weak  solution to the system (\ref{sto. NS}).
\begin{definition}\label{def. weak solu.}
A $W^{s,p}$-valued ($s\in\mathbb R,1\leq p\leq\infty$) continuous $\mathscr{F}_t$ -adapted proccess $u=\big\lbrace u(t);{t\in[0,\infty)}\big\rbrace$
is said to be a global in time probabilistically strong and analytically weak solution to system (\ref{sto. NS}) with initial data $u_0\in L^2_\sigma$ , if $~\mathbb{P} - a.s.$ : \\
$\bf{(i)}$~~~${\rm div}~u\equiv0$ in the sense of distribution;\\
$\bf{(ii)}$~~~
$u\in L^2_{loc}\big([0,\infty);L^2\big)$ ;
\\
$\bf{(iii)}$~~~for any test function $\varphi\in C^\infty_\sigma(\mathbb{T}^d)$,
we have for each $t\in[0,\infty)$ that
\begin{equation}
	\big\langle u(t),\varphi\big\rangle = \big\langle u_0,\varphi\big\rangle\ + \int^{t}_{0}\big\langle-{\rm div}\big(u(s)\otimes u(s)\big)+\Delta u(s)~,~ \varphi\big\rangle~{\rm d}s \ +\big\langle {\rm W}(t) ,~\varphi\big\rangle~.
	\label{weak solu.}
\end{equation}
\end{definition}\

To state our main result, we first decompose $u=v+z$ with $z$ solving the linear stochastic equation:
\begin{equation}
	\left\{\begin{aligned}\label{sto. low Re system}
		&{\rm d}z(t)=\big(\Delta-{\rm I}\big) z(t){\rm d}t+\nabla P_1+{\rm dW}(t)\ \ \ \ \ t\in[0,\infty)\\
		&{\rm div}\ z(t)=0\ \ \ \ \ \ t\in[0,\infty)\\
		&z(0)= u_0,
	\end{aligned}\right.
\end{equation}
and $v$ solving the  non-linear equations:
\begin{equation}\label{random NS}
	\left\{\begin{aligned}
		&\partial_t v=-{\rm div}\Big[\big(v(t)+z(t)\big)\otimes\big(v(t)+z(t)\big)\Big]+\Delta v(t)+z(t)-\nabla p(t) \ \ \ \ t\in[0,\infty) \\
		&{\rm div}\ v(t)=0\ \ \ \ \ \ t\in[0,\infty)\\
		&v(0)= 0.
	\end{aligned}\right.
\end{equation}
By adding a new damping term, we could obtain a uniform in time bound (see Theorem \ref{Thm. z regularity}.) The following is our main result and it is proved in Section \ref{sec. proof}.

\begin{theorem}\label{prop. closeness to initial vector field}
	For any $\varepsilon>0$ , any $1\leq\alpha,r<\infty$, $1\leq p<2$, any $u_0\in L^2_\sigma$ and any smooth vector field $w\in\mathcal C^1_{0,\sigma}$ ,
	there exists a probabilistically strong and analytically weak global solution $u$ to (\ref{sto. NS}) {with initial data $u_0$ ,} such that 
	\begin{align}\label{additional regularities of solution}
		u-z\in~&\bar L^2\big(\Omega,L_s^2L^2\big)
		\bigcap\bar L^{\alpha}\big(\Omega;Z^{p,r}\big)\bigcap\mathbf E_p	
	\end{align}
	and $u$ is close to $w+z$ in the following sense:
	\begin{align}\label{small deff.}
	    \big\|u-(w+z)\big\|_{\bar L^{\alpha}\left(\Omega;Z^{p,r}\right)}+\big\|u-(w+z)\big\|_{\mathbf E_p}<\varepsilon~.
	\end{align}
In particular, there exists infinitely many different probabilistically strong and analytically weak global solutions to (\ref{sto. NS}).
For the definition of spaces we refer to {(\ref{def. sup-C^infty space}),} (\ref{def. Z-space}), (\ref{def. [L,L,Y]-spaces}) and (\ref{def. E_p space}).
\end{theorem}

\begin{remark} It is well-known that the martingale solutions constructed by Galerkin approximation satisfy the energy inequality:
\begin{align}
	\bigg[\mathbb E \Big(\|u(t)\|^{2}_{L^2}\Big)\bigg]^{1/2}
	&\leq\bigg(\|u_0\|^2_{L^2}+T\cdot{{\rm Tr}}\big[G^*G\big]\bigg)^{1/2}~,\nonumber\\
	&\leq \|u_0\|_{L^2}+T^{1/2}{{\rm Tr}}\big[G^*G\big]^{1/2}~,~~~\forall t\in[0, T].\nonumber
\end{align}
However, for $d\geq3$, due to lack of uniqueness, such solutions are not  probabilistically strong solutions. By choosing different special $w$ and using the $\mathbf E_p$-closeness in (\ref{small deff.}), we can obtain infinitely many different solutions that break  the energy inequality {on $[0,T]$, i.e.} 	
\begin{equation}\label{ill energy ineq.}
	\left\lbrace\int_{{0}}^{{T}}\bigg[\mathbb E \Big(\|u(t)\|^{2}_{L^2}\Big) \bigg]^{p/2}dt\right\rbrace^{1/p}>T^{1/p}\bigg(\|u_0\|_{L^2}+T^{1/2}{\rm Tr}\big[G^*G\big]^{1/2}\bigg)~.
	\end{equation}
\end{remark}\
	
The above result implies there exists infinitely many solutions such that $u-z\in L_s^pL^\infty$. If $z$ also have such regularity we could obtain infinitely many solutions $u\in L_s^pL^\infty$. The result then reads as follows and its proof is given in Sestion \ref{sec. proof}.
\begin{corollary}\label{corrola. further regularities} Let $1\leq\alpha<\infty$, $1\leq p<2$.
	 Assume that ${\rm Tr}\big[G^*(I-\Delta)^\lambda G\big]<\infty$ and $u_0\in L^{p_1},p_1\geq2$ with
	\begin{align}\label{further regularity scalings}
		\lambda>\frac{d}{2}-1\ \ \ \text{and} \ \ \ p_1>\frac{pd}{2},	
	\end{align}
	then there exist infinitely many different probabilistically strong and analytically weak global solutions to (\ref{sto. NS}) in $\bar L^\alpha\big(\Omega,L_s^pL^\infty\big)$ with the same initial data $u_0$ .
\end{corollary}

\par The non-uniqueness of regularity $L^p_{s}L^\infty,1\leq p<2$ is sharp in the sense that the solution is unique in the space $L^2([0,T];L^\infty)$. This result is a stochastic extension of Theorem \ref{Thm. weak strong uniqueness} and its proof is given in Section \ref{sec. sto. weak-strong uniqueness}.

\begin{theorem}\label{thm. Stochastic Serrin's Criterion}
	{\bf (Pathwise Uniqueness for the Stochastic N-S System)} Let $d\geq2$,  and $0<T<\infty$ be arbitrarily fixed. There exists at most one solution $u$ to \eqref{sto. NS} satisfying 
	\begin{align}\label{Serrin's criterion}
		u\in X^{p,q}_T,~~\mathbb P-a.s.
	\end{align}
	for some $p\in[2,\infty]$ and $q\in(2,\infty]$ 
such that
	\begin{align}\label{Serrin scales}
		\frac{2}{p}+\frac{d}{q}\leq1~.
	\end{align}
	Moreover, the solution $u$ is Leray-Hopf and  $u\in L^2\Big(\Omega;C_{[0,T]}L^2\Big)\bigcap L^2\Big(\Omega;L^2_{[0,T]}H^1\Big)$
	 and satisfies the energy inequality:
	\begin{align}\label{expectational energy inequality (normal)}
		\frac{1}{2}\mathbb E\left\|u(t)\right\|_{L^2}^2+\int_0^t\mathbb E\left\|\nabla u(s)\right\|_{L^2}^2ds\leq\frac{1}{2}\mathbb E\left\|u_0\right\|_{L^2}^2+\frac{t}{2}{\rm Tr}\big[G^*G\big]~,~t\in[0,T].
	\end{align}
\end{theorem}

\par 
\subsection{Application to the stochastic Euler equations}
Theorem \ref{prop. closeness to initial vector field} also holds for the following stochastic Euler equation:
\begin{equation}
	\left\{\begin{aligned}
		&{\rm d}u(t)=\Big(-{\rm div}\big(u(t)\otimes u(t)\big)-\nabla p(t)\Big){\rm d}t+{\rm dW}(t),\\
		&{\rm div}\ u(t)=0,\\
        &u(0)=u_0.
	\end{aligned}\right.
	\label{sto. el}
\end{equation}

\begin{theorem}
	Under the same assumption as Theorem \ref{prop. closeness to initial vector field}, there exists infinitely many different probabilistically strong and analytically weak solutions to \eqref{sto. el} with the same given {$L^2$-}initial data.
\end{theorem}

\par In the two-dimensional case stochastic Euler equations have been studied in \cite{Be99, BF99, BFM16, BP01, GHV14}. The three-dimensional case has been treated in \cite{CFH19, GHV14, Ki09, MV00, HZZ21b}. In particular, Glatt-Holtz and Vicol \cite{GHV14} obtained local well-posedness of strong solutions to stochastic Euler equations in two and three dimensions, global well-posedness in two dimensions for additive and linear multiplicative noise.  Hofmanov\'{a}, Zhu and the third named author \cite{HZZ21b} established existence and non-uniqueness of global dissipative martingale solutions for additive noise. We emphasize that before our work there's no result for the existence of probabilistically strong solutions to the stochastic Euler equations in higher dimensions. Our result is the first one on this point.

Moreover, by the interpolation $L^{3/2-}_{s}C^{1/3}\supset L^p_{s}L^\infty\cap L^1_{s}C^{1-}$ for some  $1\leq p<2$, we obtain as a byproduct the existence of infinitely many non-conserving solutions in $L^{\alpha}\Big(\Omega;L^{3/2-}_{loc}\big([0,\infty);C^{1/3}\big)\Big)\bigcap$ $L^{\alpha}\Big(\Omega;L^1_{loc}\big([0,\infty);C^{1-}\big)\Big)$ to the stochastic Euler system \eqref{sto. el}. This gives the first stochastic version of the Onsagar's conjecture in negative direction with an exact "$1/3$-H\"older regularity" in spacial variables.

\begin{definition}
	A probabilistically strong and analytically weak solution (in the sense of Definition \ref{def. weak solu.} without the Laplacian term) $u$ to (\ref{sto. el}) is said to be conserving, if $~\mathbb P-a.s.$ , $$\left\|u(t)\right\|_{L^2}=\|u_0\|_{L^2}{+2\int_0^t\big\langle u(t),{\rm d}W(t)\big\rangle_{L^2}+t\cdot{\rm Tr}\big[G^*G\big]}\ \ \ \ \  \text{for all}\ \  t\in [0,\infty).$$ Otherwise, we say that the solution $u$ is non-conserving.
\end{definition}

\begin{theorem}
	Let $d\geq2$, $1\leq\alpha<\infty$, $\varepsilon>0$ and $u_0\in W^{1,\infty}$ be arbitrarily given. And let $G$ satisfies ${\rm Tr}\big[G^*(I-\Delta)^{{\frac{d}{2}+1}}G\big]<\infty$ . Then there exist infinitely many {different non-conserving probabilistically strong and analytically weak solutions} $L^\alpha\Big(\Omega;L^{3/2-\varepsilon}_{loc}\big([0,\infty);C^{1/3}\big)\Big)\bigcap L^\alpha\Big(\Omega;L^1_{loc}\big([0,\infty);C^{1-\varepsilon}\big)\Big)$ to the stochastic Euler system (\ref{sto. el}){, with initial data $u_0$} .
\end{theorem}

\subsection{Notations}
Throughout the paper, we employ the notation $a\lesssim b$ if there exists a constant $c>0$ such that $a\leq cb$. $\mathbb N_0=\mathbb{N}\cup\{0\}$.
Let $\mathcal{S}'$ be the space of distributions on $\mathbb{T}^d$. For $s\in\mathbb R$ and $1\leq p\leq\infty$, the Sobolev space $W^{s,p}=\big\{f\in\mathcal{S}':\big\|f\big\|_{W^{s,p}}:=\big\|(I-\Delta)^{s/2}f\big\|_{L^p}<\infty\big\}$. 

\par Given a Banach space $\Big(Y,~\|\cdot\|_Y\Big)$ and $I\subset \mathbb{R}$ we write $L^p(I;Y)$ for the space of $L^p$-integrable functions from $I$ to $\mathbb{R}$, equipped with the usual $L^p$-norm. We also use $L^p_IY$ to denote the space of functions $f$ from $[0,\infty)$ to $Y$ satisfying $f|_I\in L^p(I,Y)$. We also use $L^p_{loc}([0,\infty);Y)$ to denote the space of functions $f$ from $[0,\infty)$ to $Y$ satisfying $f\in L^p_{[0,T]}Y$ for all $T>0$. We also write $C(I;Y)$ for the space of continuous functions from $I$ to $Y$ equipped with the supremum norm in a bounded subset. $C_IY$ is similar as $L^p_IY$  with $f|_I\in C(I,Y)$. We also use $L^\alpha(\Omega,Y)$ to denote the space of functions on $\Omega$ with finite $\alpha$ moment, equipped with the usual $L^\alpha$-norm.
  Whenever $I=[s,s+1]$, we simply write $L^p_sY:=L^p_{[s,s+1]}Y$ and $C_sY:=C_{[s,s+1]}Y$.
\par For smooth tensor fields, we use the following notations:
\begin{align}
	&C^k_{c,\sigma}:=C^k_{c,\sigma}\left([0,\infty)\times\mathbb T^d\right),~~C^k_{\sigma}:=C^k_{\sigma}\left([0,\infty)\times\mathbb T^d\right),~~C^k_{t,x}:=C^k_{}\left([0,\infty)\times\mathbb T^d\right)\nonumber
\end{align}
where $k\in\mathbb N_0\cup\lbrace\infty\rbrace$ and the indices "$c$" and "$\sigma$" mean "compact support in time" and "divergence-free", respectively. And for $C^N$-norms and semi-norms of function $f\in C^N_{t,x}$, we write
 \begin{align}
 	\big\lbrack f(t)\big\rbrack_m &:=\sum_{|\alpha|=m,\alpha\in\mathbb{N}_0^3}\sup_{x\in\mathbb T^d}\Big|\partial_x^\alpha f(t,x)\Big|~~~,~~~t\in\mathbb R~;\nonumber\\
 	\big\lbrack f\big\rbrack_{m,[a,b]} &:=\sum_{|\alpha|+k=m,\alpha\in\mathbb{N}_0^3,k\in\mathbb{N}_0}\sup_{(t,x)\in[a,b]\times\mathbb T^d}\Big|\partial_t^k\partial_x^\alpha f(t,x)\Big|~~~,~~~-\infty<a<b<\infty~;\nonumber
 \end{align}
 for $0\leq m\leq N$, and
 \begin{align}
 	\big\|f(t)\big\|_N
 	&:=\sum_{m=0}^N\big\lbrack f(t)\big\rbrack_m~~~,~~~t\in\mathbb R~;\nonumber\\
 	\big\|f\big\|_{N,[a,b]}
 	&:=\sum_{m=0}^N\big\lbrack f\big\rbrack_{m,[a,b]}~~~,~~~-\infty<a<b<\infty~.\nonumber
 \end{align}
 For simplicity, we  write $\big\lbrack f\big\rbrack_{m,s}:=\big\lbrack f\big\rbrack_{m,[s,s+1]}$ and $\big\|f\big\|_{N,s}:=\big\|f\big\|_{N,[s,s+1]}$ for $s\geq0$.
\\


We define the spaces for $1\leq p<2$ and $\alpha, r\geq1$ :
\begin{align}\label{def. sup-C^infty space}
	\mathcal C^1_{0,\sigma}:=\left\lbrace u\in C^1_{\sigma}~\Big|~\big\|u\big\|_{C^1_{\sigma}\left([0,\infty)\times\mathbb T^d\right)}<\infty~,u(0)=0\right\rbrace~;\ \ \ \ \ \ \ \ \ \ \ \ \ \ \ \ \ \ \
\end{align}
 \begin{align}
 	Z^{p,r}_{s}
 	&:=C_{s}W^{-1,1}\bigcap L^p_{s}L^\infty\bigcap L^1_{s}W^{1,r},\nonumber\\
 	\textrm{with}&\quad\|u\|_{{p,r};s}:=\|u\|_{C_{s}W^{-1,1}}+\|u\|_{L^p_{s}L^\infty}+\|u\|_{L^1_{s}W^{1,r}},~s\geq0~;
 \end{align}
\begin{align}\label{def. Z-space}
	\bar L^\alpha(\Omega,Z^{p,r})
 	&:=\left\lbrace u:\Omega\times[0,\infty)\longrightarrow W^{-1,1}\bigcap L^\infty\bigcap W^{1,r} ~\bigg|~~~\big\|u\big\|_{\bar L^\alpha(\Omega,Z^{p,r})}:=\sup_{s\geq0}\left(\mathbb E\big\|u\big\|_{p,r;s}^\alpha\right)^{1/\alpha}<\infty\right\rbrace .
\end{align}
For stochastic processes $u:\Omega\times[0,\infty)\longrightarrow Y$, we introduce the following notations for $\alpha,p\geq1$ :
 \begin{align}
 	\bar L^\alpha(\Omega,L^p_sY)
 	&:=\left\lbrace u:\Omega\times[0,\infty)\longrightarrow Y~\bigg|~~~\big\|u\big\|_{\bar L^\alpha(\Omega,L^p_sY)}:=\sup_{s\geq0}\left(\mathbb E\big\|u\big\|_{L^p_{s}Y}^\alpha\right)^{1/\alpha}<\infty\right\rbrace ,\label{def. [L,L,Y]-spaces}\\
 	\bar L^\alpha(\Omega,C_sY)
 	&:=\left\lbrace u:\Omega\times[0,\infty)\longrightarrow Y~\bigg|~~~\big\|u\big\|_{\bar L^\alpha(\Omega,C_sY)}:=\sup_{s\geq0}\left(\mathbb E\big\|u\big\|_{C_{s}Y}^\alpha\right)^{1/\alpha}<\infty\right\rbrace ,\nonumber\\
 	\mathbf E_p&:=\left\lbrace u:\Omega\times[0,\infty)\longrightarrow Y~\bigg|~~~\big\|u\big\|_{\mathbf E_p}:=\sup_{s\geq0}\big\|u\big\|_{L^p_sL^2\left(\Omega;~L^\infty\right)}<\infty\right\rbrace .\label{def. E_p space}
 \end{align}
  It's easy to see that $\bar L^\alpha(\Omega,L^p_sY)\subset L^\alpha\Big(\Omega,L^p_{loc}\big([0,\infty);Y\big)\Big)$ and $\bar L^\alpha(\Omega,C_sY)
 	\subset L^\alpha\Big(\Omega,C\big([0,\infty);Y\big)\Big)$.
\

\subsection{Organization of the paper}
In Section 2, we give the proof of Theorem \ref{prop. closeness to initial vector field} and Corollary \ref{corrola. further regularities}  assuming the main iteration Proposition \ref{prop. main iteration}. The proof of Proposition \ref{prop. main iteration} is given in Section \ref{sec. Construction of the Iteration} and \ref{sec. Inductive Estimates}: We construct the velocity perturbation and the new Reynold Stress  in Section \ref{sec. Construction of the Iteration}. Estimate of the perturbation and Reynold stress error is presented in Section \ref{sec. Inductive Estimates}. Section \ref{sec. sto. weak-strong uniqueness} is devoted to the proof of Theorem \ref{thm. Stochastic Serrin's Criterion}. In  Appendix we collect several auxiliary results.

\section{Proof of Theorem \ref{prop. closeness to initial vector field}} \label{sec. proof}
This section is devoted to the proof of Theorem \ref{prop. closeness to initial vector field} and Corollary \ref{corrola. further regularities}. More precisely, by means of the convex integration method we construct infinitely many global-in-time probablistically strong solutions to the Navier-Stokes system \eqref{sto. NS}. Different from the previous work using convex integration in stochastic setting by introducing suitable stopping times, we take expectation in the convex integration iterative estimates and construct directly solutions on the whole time interval $[0,\infty)$. The key point is that the $m$-th moment of approximate velocity and the error at step $q$ only depends on $m$ and the parameters up to the step $q$ of the iteration. Then we could choose the parameters at the level $q+1$ to guarantee smallness of the velocity perturbations and the error at step $q+1$.
\par As the first step, we recall the following regularity result for the linear equation. The linear system (\ref{sto. low Re system}) is known to be well-posed and the solution $z$ can be represented as follows:
\begin{equation}\label{z represeantation}
	z(t)=e^{t(\Delta-{\rm I})}u_0+W_{con}(t)~,\ \ t\in[0,\infty) ~,
\end{equation}
where $W_{con}(t):=\mathbf{P}\int_0^t e^{(t-s)(\Delta-{\rm I})}{{\rm dW}(s)}$ with the Helmholtz projection $\mathbf{P}$.
The following result is well known and can be obtained by using the method in \cite{Da04}.

\begin{theorem}\label{Thm. z regularity}
Let $W_{con}(t):=\mathbf{P}\int_0^t e^{(t-s)(\nu\Delta-{\rm I})}{{\rm dW}(s)}$ with $\nu=0,1$. For $0<\delta<1/2$ and $1\leq m<\infty$, there exists finite constant $C_{m,\delta}>0$   such that
\begin{equation}\label{z finiteness of moments}
	\sup_{s\geq0}\mathbb E\Big(\big\|W_{con}\big\|^m_{C^{{1/2-\delta}}_{[s,s+1]}L^2}\Big)\leq{C_{m,\delta}}.
\end{equation}
\end{theorem}


\par Let us now explain how the convex integration is set up. More precisely, we intend to develop an iteration procedure leading to the proof of Theorem \ref{prop. closeness to initial vector field}. The iteration is indexed by a parameter $q\in\mathbb{N}_0$. At each step $q$ , a pair $\left(v^{(q)},\mathring{R}^{(q)}\right)$ is constructed solving the following system
\begin{equation}\label{random Rey.}
	\left\{\begin{aligned}
		&\partial_tv^{(q)}+{\rm div}\left[\big(v^{(q)}+z\big)\otimes \big(v^{(q)}+z\big)\right]=-\nabla p^{(q)}+\Delta v^{(q)}+z+{\rm div}\mathring{R}^{(q)},\\
		&{\rm div}\ v^{(q)}\equiv 0,\\
		&v^{(q)}(0)=0.\\
	\end{aligned}\right.
\end{equation}
Here the trace-free $d\times d$ matrix $\mathring{R}^{(q)}(t, x)$ is the so-called Reynold stress term.

The iteration starts at
\begin{align}\label{start point of the iteration}
	v^{(0)}=w,~~~\mathring{R}^{(0)}=\mathcal{R}\big(\partial_tw-\Delta w-z\big)+(w+z)~\mathring{\otimes}~(w+z)~,
\end{align}
where $w\in{\mathcal C^{1}_{0,\sigma}}$ is the pre-given vector field as in Proposition \ref{prop. closeness to initial vector field} and $\mathcal{R}$ denotes the reverse-divergence operator which we recall in Appendix \ref{Appendix B} for convenience. Note that by definition, $v^{(0)}(0)=w(0,\cdot)=0$. By (\ref{z finiteness of moments}), we have $\mathbb{P}-a.s.$ $\mathring{R}^{(0)}\in C_{s}L^1$ for all $s\geq0$ and
\begin{align}
	\left\|\mathring{R}^{(0)}\right\|_{\bar L^1(\Omega,L_s^1L^1)}
	&\lesssim{\sup_{s\geq0}\big\|\partial_t w\big\|_{L_s^1L^1}+\sup_{s\geq0}\big\|w\big\|_{L_s^1H^1}+\big\|z\big\|_{\bar L^2(\Omega,L_s^2L^2)}+\sup_{s\geq0}\big\|w\big\|_{L_s^2L^2}^2+\big\|z\big\|_{\bar L^2(\Omega,L_s^2L^2)}^2}\nonumber\\
	&{\lesssim1+\sup_{s\geq0}\big\|w\big\|_{1,s}^2+\big\|z\big\|_{\bar L^2(\Omega,L_s^2L^2)}^2}<\infty~.\nonumber
\end{align}
The main ingredient in the proof of Theorem \ref{prop. closeness to initial vector field} is the following iteration.

\begin{proposition}\label{prop. main iteration}
	{\bf (Main Iteration)} Let $1\leq\alpha,r<\infty$, $1\leq p<2$, $u_0\in L^2_\sigma$. 
Let also $\delta>0$ be arbitrarily given. If  $(v^{(q)},\mathring{R}^{(q)}) \in C^\infty_\sigma\times C_{[0,\infty)}L^1$ is an $(\mathcal{F}_t)_{t\geq0}$-adapted solution to \eqref{random Rey.}, then there exists an $(\mathcal{F}_t)_{t\geq0}$-adapted process $(v^{(q+1)}, \mathring{R}^{(q+1)})\in C^\infty_\sigma\times C_{[0,\infty)}L^1$ which solves \eqref{random Rey.} and satisfies
	\begin{equation}\label{smallness of Rey.}
		\left\|\mathring{R}^{(q+1)}\right\|^{1/2}_{\bar L^1(\Omega,L^1_sL^1)}\leq\delta~;
	\end{equation}
	 and 
	\begin{equation}\label{L2 smallness of pert.}
	\big\|v^{(q+1)}-v^{(q)}\big\|_{\bar L^2(\Omega,L^2_sL^2)}\leq C\left(\left\|\mathring{R}^{(q)}\right\|^{1/2}_{\bar L^1(\Omega,L^1_sL^1)}+2^{-q}\right)+\delta~
	\end{equation}
	with some constant $C>0$ independent of $q$ , and
	\begin{align}\label{weak smallness of pert.}
		\big\|v^{(q+1)}-v^{(q)}\big\|_{\bar L^{\alpha}\left(\Omega,Z^{p,r}\right)}+\big\|v^{(q+1)}-v^{(q)}\big\|_{\mathbf E_p}
	   \leq\delta~.
	\end{align}
\end{proposition}\
\par The proof of this result is presented in Section \ref{sec. Construction of the Iteration}\ and  \ref{sec. Inductive Estimates} below. Now we have all in hand to complete the proof of Theorem \ref{prop. closeness to initial vector field} and Corollary \ref{corrola. further regularities}.

\begin{proof}[Proof of Theorem \ref{prop. closeness to initial vector field}]
	Let $1\leq\alpha,r<\infty,1\leq p<2$, the initial data $u_0\in L^2_\sigma$, the smooth vector field $w\in\mathcal C^1_{0,\sigma}$ and $\varepsilon>0$ be arbitrarily given as stated in Theorem \ref{prop. closeness to initial vector field}. We start the iteration from (\ref{start point of the iteration}). We repeatedly apply Proposition \ref{prop. main iteration} with
	\begin{equation}\label{step-length of the iteration}
		\delta_{q+1}:={\rm min}\left\{\left\|\mathring{R}^{(q)}\right\|^{1/2}_{\bar L^1(\Omega,L_s^1L^1)},\frac{\varepsilon}{2^{q+1}}\right\},\quad q\in\mathbb{N}_0,
	\end{equation}
and obtain $(\mathcal{F}_t)_{t\geq0}$-adapted process $(v^{(q)}, \mathring{R}^{(q)})\in C^\infty_\sigma\times C_{[0,\infty)}L^1$ such that
	 for all $q\geq1$ that
	\begin{equation}
		\left\|\mathring{R}^{(q)}\right\|_{\bar L^1(\Omega,L_s^1L^1)}\leq\delta_q^2~~,\label{rq}
	\end{equation}
	and for $\varpi_{q}=v^{(q)}-v^{(q-1)}$
	\begin{align}\label{smallness of perturbations in iteration}
		\big\|\varpi_{q}\big\|_{\bar L^{\alpha}\left(\Omega,Z^{p,r}\right)}+
		\big\|\varpi_{q}\big\|_{\mathbf E_p}\leq\delta_q~;
	\end{align}
	and
	\begin{equation}
		\big\|\varpi_{q}\big\|_{\bar L^2(\Omega,L_s^2L^2)}\lesssim\left\|\mathring{R}^{(q-1)}\right\|^{1/2}_{\bar L^1(\Omega,L_s^1L^1)}+2^{-(q-1)}~.\nonumber
	\end{equation}
	Here the implicit constant is deterministic and independent of $q$ and $\varepsilon$.
Moreover, by \eqref{rq} for all $q\geq2$
\begin{equation}\label{only for explanation 2}
		\big\|\varpi_{q}\big\|_{\bar L^2(\Omega,L_s^2L^2)}\lesssim\delta_{q-1}+2^{-(q-1)}~.
	\end{equation}
Hence, $\mathring{R}^{(q)}\rightarrow0$ and there exists some $v\in {\bar L^2\big(\Omega,L_s^2L^2\big)\bigcap}\bar L^\alpha\big(\Omega,Z^{p,r}\big)\bigcap\mathbf E_p$ such that
	\begin{align}
				v^{(q)}
		&\longrightarrow v\ \ \ in\ \ \ \bar L^2\big(\Omega,L_s^2L^2\big)\bigcap\bar L^\alpha\big(\Omega,Z^{p,r}\big)\bigcap\mathbf E_p~.\nonumber
	\end{align}
Taking the limit as $q\rightarrow \infty$,
one sees that $v$ satisfies the system (\ref{random NS}). Thus the process $u:=v+z$ is a solution to (\ref{sto. NS}) with initial data $u_0$ in the sense of Definition \ref{def. weak solu.}.
\par Finally, by (\ref{step-length of the iteration}) and (\ref{smallness of perturbations in iteration}), we have the  closeness in $L^{\alpha}\Big(\Omega;Z^{p,r}\Big)$ and $\mathbf E_p$-norms :
\begin{align}
	\big\|u-(w+z)\big\|_{\bar L^{\alpha}\left(\Omega,Z^{p,r}\right)}+\big\|u-(w+z)\big\|_{\mathbf E_p}
	&=\big\|v-w\big\|_{\bar L^{\alpha}\left(\Omega,Z^{p,r}\right)}+\big\|v-w\big\|_{\mathbf E_p}\nonumber\\
	&\leq\sum_{q=1}^{\infty}\Big(\big\|\varpi_{q}\big\|_{\bar L^{\alpha}\left(\Omega,Z^{p,r}\right)}+\big\|\varpi_{q}\big\|_{\mathbf E_p}\Big)\nonumber\\
	&\leq\sum_{q=1}^{\infty}\frac{\varepsilon}{2^{q}}<\varepsilon~.\nonumber
\end{align}
Thus the proof is complete.
\end{proof}\

\begin{proof}[Proof of Corollary \ref{corrola. further regularities}]
{First we choose a sequence $\lbrace w_n; n\in\mathbb N\rbrace\subset\mathcal C^{\infty}_{0,\sigma}$ such that
\begin{align}\label{w_n choice (ii)}
		\sup_{s\geq0}\big\|w_i-w_j\big\|_{L^p_sL^\infty}\geq1~,\text{~~for~each~pair~}i\neq j~.
\end{align}
Applying Theorem \ref{prop. closeness to initial vector field} to $u_0$, $\varepsilon=1/3$ and each $w_n$ respectively then yields a sequence $\lbrace u_n\rbrace$ of solutions to (\ref{sto. NS}) such that $u_n-z\in\bar L^2\big(\Omega,L_s^2L^2\big)\bigcap\bar L^{\alpha}\big(\Omega;Z^{p,r}\big)\bigcap\mathbf E_p	$ and
\begin{align}
	\big\|u_n-(w_n+z)\big\|_{\bar L^\alpha\left(\Omega;Z^{p,r}\right)}+\big\|u_n-(w_n+z)\big\|_{\mathbf E_p}<1/3~.
\end{align}
Then any two of the solutions are different:
\begin{align}
	&~~~~~\big\|u_i-u_j\big\|_{\bar L^\alpha(\Omega,~L^p_sL^\infty)}\nonumber\\
	    &\geq\sup_{s\geq0}\big\|w_i-w_j\big\|_{L^p_sL^\infty}-\big\|u_i-(w_i+z)\big\|_{\bar L^\alpha(\Omega,~L^p_sL^\infty)}-\big\|u_j-(w_j+z)\big\|_{\bar L^\alpha(\Omega,~L^p_sL^\infty)}\nonumber\\
	    &\geq1-1/3-1/3=1/3~,~~~\forall\  i\neq j~.
\end{align}}

Now we show that $u_n\in\bar L^\alpha\big(\Omega,L_s^pL^\infty\big)$. It suffices to prove $z\in\bar L^\alpha\big(\Omega,L_s^pL^\infty\big)$. For this, we split $z$ into two parts by (\ref{z represeantation}). For the first part, by standard estimates of heat kernel, we have for $2\leq p_1<\infty$ that
\begin{align}
	\left\|e^{t(\Delta-I)}u_0\right\|_{L^{\infty}}\lesssim e^{-\frac{t}{2}}t^{-\frac{d}{2p_1}}\big\|u_0\big\|_{L^{p_1}}~,\ \ \ t\in(0,\infty)~.\nonumber
\end{align}
Hence, for $e^{t(\Delta-I)}u_0\in\bar L^\alpha\big(\Omega,L_s^pL^\infty\big)$, we only need $1-\frac{pd}{2p_1}>0$~, i.e. $p_1>\frac{pd}{2}$~. For the stochastic convolution part, note that ${\rm Tr}\big[G^*(I-\Delta)^\lambda G\big]<\infty$ implies $W_{con}\in \bar L^\alpha\big(\Omega,C_sL^\infty\big)\subset\bar L^\alpha\big(\Omega,L_s^pL^\infty\big)$ for all $1\leq\alpha<\infty$.
\end{proof}\

\section{Construction of the Iteration}\label{sec. Construction of the Iteration}\

\par This section is devoted to the construction of $v^{(q+1)}$ and  $\mathring R^{(q+1)}$. To this end, we employ a two-step approach. In the first step we do mollification to avoid a loss of derivative during convex integration scheme. In the second step we  proceeds along  the lines of \cite[Section 4]{CL22} and do space-time convex integration. To keep the initial condition for the mollification equation we introduce a time cutoff function for the perturbation which leads to an extra error (see \cite{BMS21,HZZ21a}).

\subsection{Mollification}
\noindent\par We intend to replace $(v_q,\mathring R^{(q)})$ by a mollified field $(v^{(q)}_\ell,\mathring{R}^{(q)}_\ell)$. To this end, let $\phi\in C^\infty_c\left(\mathbb R^d;\mathbb R_+\right)$ with $supp~\phi\subset B_1(0)$, and $\varphi\in C^\infty_c\left(\mathbb R;\mathbb R_+\right)$ with $supp~\varphi\subset[0,1]$. And we define the mollifiers as follows:
\begin{align}
	&\phi_{\ell_{q+1}}:=\ell_{q+1}^{-d}\phi\left(\cdot/\ell_{q+1}\right)~,\nonumber\\
	&\varphi_{\ell_{q+1}}:=\ell_{q+1}^{-1}\varphi\left(\cdot/\ell_{q+1}\right)~.\nonumber
\end{align}
The one sided mollifier here is used in order to preserve adaptedness. Here, $\ell_{q+1}\in(0,1)$ is a small parameter and will be set later. If there's no confusion, we will simply write $\ell$ for $\ell_{q+1}$. Now we extend $v^{(q)},z,p_q$ and $\mathring R^{(q)}$ to $t<0$ by taking them equal to the value at $t=0$. Then  $v^{(q)},z,p_q$ and $\mathring R^{(q)}$ also satisfies equation for $t<0$ as $\partial_t v^{(q)}=0$ from construction.
\par We define a mollification of $v^{(q)},\mathring{R}^{(q)},z$ in space and time by convolution as follows:
\begin{align}
	&~~~~z_\ell:=\left(z\ast_x\phi_\ell\right)\ast_t\varphi_\ell~,\nonumber\\
	&~v^{(q)}_\ell:=\left(v^{(q)}\ast_x\phi_\ell\right)\ast_t\varphi_\ell~,\nonumber\\
	&\mathring{R}^{(q)}_\ell:=\left(\mathring{R}^{(q)}\ast_x\phi_\ell\right)\ast_t\varphi_\ell~.\nonumber
\end{align}
Then a slight calculation shows that $\left(v^{(q)}_\ell, \mathring{R}^{(q)}_\ell\right)$ satisfies on $[0,\infty)\times\mathbb T^d$
\begin{equation}\label{mollified Reynold system}
	\left\{\begin{aligned}
		&\partial_tv^{(q)}_\ell+{\rm div}\left[\left(v^{(q)}_\ell+z_\ell\right)\otimes \left(v^{(q)}_\ell+z_\ell\right)\right]=-\nabla {\hat p}_\ell^{(q)}+\nu \Delta v^{(q)}_\ell+z_\ell+{\rm div}\left(\mathring{R}^{(q)}_\ell+\mathring{R}^\ell_{com}\right)~, \\
		&{\rm div}\ v^{(q)}_\ell\equiv 0~,\\
&v^{(q)}_\ell(0)=0~,
	\end{aligned}\right.
\end{equation}
with  $\mathring{R}^\ell_{com}$ and  ${\hat p}_\ell^{(q)}$ given by
\begin{equation}\label{definition of molli-communication Reynold stress}
	\mathring{R}^\ell_{com}:=\left(v^{(q)}_\ell+z_\ell\right)\mathring\otimes \left(v^{(q)}_\ell+z_\ell\right)-\left[\left((v^{(q)}+z)\mathring\otimes (v^{(q)}+z)\right)\ast_x\phi_\ell\right]\ast_t\varphi_\ell~,
\end{equation}
\begin{equation}
	{\hat p}_\ell^{(q)}:=p_\ell^{(q)}-\frac{1}{d}\left|v^{(q)}_\ell+z_\ell\right|^2+\frac{1}{d}\left(\left|v^{(q)}+z\right|^2 \ast_x\phi_\ell\right)\ast_t\varphi_\ell~.\nonumber
\end{equation}
It is easy to see that $z_\ell$ is $(\mathcal{F}_t)_{t\geq0}$-adapted and so are $v^{(q)}_\ell$ and $\mathring{R}^{(q)}_\ell.$

\subsection{Construction of the Main Perturbation $\omega_{q+1}$}
\indent\par Let us now proceed with the construction of the perturbation $\omega_{q+1}$ which then defines the next iteration by $v^{(q+1)}:=v^{(q)}_\ell+\omega_{q+1}$. To this end we employ the stationary Mikado flows introduced in \cite{DS17} and presented in \cite{CL22}, which we recall in Appendix \ref{Appendix A}. In particular, the building blocks $\mathbb W_k^{(q+1)}$ ($k\in\Lambda$) is given by \eqref{definition for W_k} with $\mu=\mu_{q+1}$, the spatial concentration parameter whose value will be given in Section \ref{subsec. setting of iteration parameters}.
Here, $\Lambda\subset\mathbb Z^d$ introduced in Lemma \ref{lem. Geometry} is a finite set. Now we introduce a spacial oscillation parameter $\sigma_{q+1}\in\mathbb N$ and $\mathbb W_k^{(q+1)}\big(\sigma_{q+1}x\big)$ is $\left(\sigma_{q+1}^{-1}\mathbb T\right)^d$-periodic.
\par Following \cite{CL22}, we then use a temporal smooth function to oscillate the building blocks intermittently in time. We choose a function $g\in C_c^\infty\big((0,1)\big)$ with $\|g\|_{L^2}=1$, and define
\begin{equation}\label{definition of g_kappa}
	g_{\kappa}(t):=\kappa_{q+1}^{1/2}\cdot g\big(\kappa_{q+1}t\big)~,~~t\in[0,1]~.
\end{equation}
Here, $\kappa_{q+1}>0$ is a large constant and will be specified later. Next, we extend the function $g_{\kappa}$ periodically onto {$[0,\infty)$} and still denote it by $g_{\kappa}$. One can show that for all $1\leq m\leq\infty$ and all $s\in[0,\infty)$~,
\begin{align}\label{L^alpha estimates for g_kappa}
	&\big\|g_{\kappa}\big\|_{L^m[s,s+1]}\lesssim\kappa_{q+1}^{1/2-1/m}~,\\
	\label{L^2 norm for g_kappa}
	&\big\|g_{\kappa}\big\|_{L^2[s,s+1]}=1~.
\end{align}
We also introduce a temporal oscillation parameter $\varsigma_{q+1}\in\mathbb N$  so that the rescaled function $g_{\kappa}\big(\varsigma_{q+1}\cdot\big)$ is $\varsigma_{q+1}^{-1}$-periodic.
\par The parameters $\mu_{q+1}$, $\sigma_{q+1}$, $\kappa_{q+1}$ and $\varsigma_{q+1}$ are assumed to be sufficiently large for the moment and will be set  in Section \ref{subsec. setting of iteration parameters}.


\par As the next step, we shall define certain amplitude functions used in the definition of the perturbation $\omega_{q+1}$. To this end, let $\chi_{q+1}\in C^\infty\left(\mathbb R^{d\times d};\mathbb R_+\right)$ be such that
\begin{align}\label{definition for the cut-off chi}
	&(i)~\chi_{q+1}\left(R\right) ~is~monotonically~increasing~with~respect~to~\left|R\right|~;\nonumber\\
	&(ii)~\chi_{q+1}\left(R\right)=\left\{\begin{aligned}
		&~~~4^{-(q+1)}~,~~~0\leq\left|R\right|\leq4^{-(q+1)}~;\\
		&~\left|R\right|~,~2\leq\left|R\right|<\infty~;
	\end{aligned}\right.\\
	&(iii)~\left|R\right|/2\leq\chi_{q+1}\left(R\right)\leq2\left|R\right|~,~when~4^{-(q+1)}<\left|R\right|<2~.\nonumber
\end{align}
It is easy to see that the function $\chi_{q+1}$ has bounded partial derivatives of all orders. Then we define  $\varrho_{q+1}$ by
\begin{equation}\label{definition for rescaling function varrho}
	\varrho_{q+1}:=4~\chi_{q+1}\left(\mathring{R}^{(q)}_\ell\right)~.
\end{equation}
It follows that $2|\mathring{R}^{(q)}_\ell|\leq \varrho_{q+1}$ and hence
 $\rm{Id}-\mathring{R}^{(q)}_\ell/\varrho_{q+1}\in B_{1/2}\left(\rm{Id}\right)$, which by Lemma \ref{lem. Geometry} and (\ref{normalized L^2-norm of Psi_k}), (\ref{definition for W_k}) implies that
\begin{align}\label{geometric decomposition of relative Reynold stress}
	\varrho_{q+1}\rm{Id}-\mathring{R}^{(q)}_\ell
	&=\sum_{k\in\Lambda}\varrho_{q+1}\Gamma_k^{2}\left(\rm{Id}-\frac{\mathring{R}^{(q)}_\ell}{\varrho_{q+1}}\right)\mathbf{e}_k\otimes\mathbf{e}_k~,\nonumber\\
	&=\sum_{k\in\Lambda}\varrho_{q+1}\Gamma_k^{2}\left(\rm{Id}-\frac{\mathring{R}^{(q)}_\ell}{\varrho_{q+1}}\right)\int_{\mathbb T^d} \mathbb W_k^{(q+1)}\big(x\big)\otimes\mathbb W_k^{(q+1)}\big(x\big){\rm d}x~\nonumber\\
    &=\sum_{k\in\Lambda}\varrho_{q+1}\Gamma_k^{2}\left(\rm{Id}-\frac{\mathring{R}^{(q)}_\ell}{\varrho_{q+1}}\right)\int_{\mathbb T^d} \mathbb W_k^{(q+1)}\big(\sigma_{q+1}x\big)\otimes\mathbb W_k^{(q+1)}\big(\sigma_{q+1}x\big){\rm d}x~.
\end{align}
Now we define the amplitude function
\begin{equation}\label{definition of a_k}
	a_k^{(q+1)}(t,x):=g_{\kappa}\left(\varsigma_{q+1}t\right)\varrho_{q+1}^{1/2}(t,x)\Gamma_k\left(\rm{Id}-\frac{\mathring{R}^{(q)}_\ell}{\varrho_{q+1}}\right)~,~~~~~k\in\Lambda~.
\end{equation}
Note that $a_k^{(q+1)}$ is $(\mathcal{F}_t)_{t\geq0}$-adapted.
With these preparation in hand, we define the principal perturbation as
\begin{equation}\label{definition of principal pertutbation}
	\omega_{q+1}^{(p)}(t,x):=\sum_{k\in\Lambda}a_k^{(q+1)}(t,x)\mathbb W_k^{(q+1)}\left(\sigma_{q+1}x\right)~.
\end{equation}

We also define the incompressibility corrector
\begin{equation}\label{definition of corrector perturbation}
	\omega_{q+1}^{(c)}(t,x):=\sigma_{q+1}^{-1}\sum_{k\in\Lambda}\nabla a_k^{(q+1)}(t,x):\mathbb V_k^{(q+1)}\left(\sigma_{q+1}x\right)~,
\end{equation}
where $\mathbb V_k^{(q+1)}$ is given in \eqref{definition for V_k} with $\mu=\mu_{q+1}$.
By \eqref{relation divV_k=W_k} and a direct computation we deduce that
\begin{equation}\label{div-representation of w^p+w^c}
	\omega_{q+1}^{(p)}+\omega_{q+1}^{(c)}=\sigma_{q+1}^{-1}{\rm div} \sum_{k\in\Lambda}a_k^{(q+1)}(t,x)\mathbb V_k^{(q+1)}\left(\sigma_{q+1}x\right)~,
\end{equation}
and hence $\textrm{div}(\omega_{q+1}^{(p)}+\omega_{q+1}^{(c)})=0$  since $a_k^{(q+1)}(t,x)\mathbb V_k^{(q+1)}$ is skew-symmetric.

Next we introduce the temporal corrector. To this end, we define the function
\begin{equation}\label{definition of h_kappa}
	h_{\kappa}(t):=\int_0^t\Big(g_{\kappa}(s)^2-1\Big)ds~,~~~~~{t\in\mathbb R}~.
\end{equation}
Then by properties of the function $g_{\kappa}$ one can easily see that $h_{\kappa}$ is periodic with  period $1$ and smooth on $[0,\infty)$ and obeys the bound
\begin{align}\label{uniform bounds for h_kappa}
	\vert h_{\kappa}(t)\vert\leq 1~~~,~~{t\in[0,\infty)}~.
\end{align}
Then the divergence-free temporal corrector is defined by
\begin{equation}\label{definition of temporal perturbation}
	\omega_{q+1}^{(t)}(t,x):=\varsigma_{q+1}^{-1}\cdot h_{\kappa}(\varsigma_{q+1}t)\mathbf P{\rm div}\mathring{R}^{(q)}_\ell(t,x)~.
\end{equation}
Here, $\mathbf P=\rm{Id}-\nabla\Delta^{-1}{\rm div}$~ is the Helmholtz projection.

To keep the initial condition, we introduce a  smooth cut-off $\Theta_{q+1}\in C^\infty\big([0,\infty);[0,1]\big)$ such that
{
\begin{align}\label{def. cutoff Theta}
	\Theta_{q+1}(t)=\left\{\begin{aligned}
		&0 \quad \textrm{when}~t\leq\ell_{q+1}^{1/2}/2~,\\
		&1 \quad \textrm{when}~t\geq\ell_{q+1}^{1/2}~,
	\end{aligned}\right.
	\ \ \ \ \ \ \rm{and}\ \ \ \ \ \ \big\|\Theta^{(n)}_{q+1}\big\|_0\lesssim\ell_{q+1}^{-n/2}~.
\end{align}}
And we define
\begin{align}
	\widetilde\omega_{q+1}^{(p)}:=\Theta_{q+1}\omega_{q+1}^{(p)}~,\quad
	\widetilde\omega_{q+1}^{(c)}:=\Theta_{q+1}\omega_{q+1}^{(c)}~,\quad
	\widetilde\omega_{q+1}^{(t)}:=\Theta_{q+1}^2\omega_{q+1}^{(t)}~.\nonumber
\end{align}
Finally, the total perturbation $\omega_{q+1}$ is defined by
\begin{align}
	\omega_{q+1}={\widetilde\omega_{q+1}^{(p)}+\widetilde\omega_{q+1}^{(c)}+\widetilde\omega_{q+1}^{(t)}}~,\nonumber
\end{align}
which is mean zero, divergence free and $(\mathcal{F}_t)_{t\geq0}$-adapted. The new velocity $v^{(q+1)}$ is defined as
\begin{equation}\label{construction of v^(q+1)}
	v^{(q+1)}:=v^{(q)}_\ell+\omega_{q+1}=v^{(q)}+\varpi_{q+1}
\end{equation}
with $\varpi_{q+1}:=v^{(q+1)}-v^{(q)}=\left(v^{(q)}_\ell-v^{(q)}\right)+\omega_{q+1}$.

\subsection{The new Reynold Stress $\mathring R^{(q+1)}$}
\indent\par In this subsection we give the  new Reynold Stress $\mathring R^{(q+1)}$.
First, according to (\ref{stationary Euler equation for W_k}), (\ref{geometric decomposition of relative Reynold stress}) and (\ref{definition of principal pertutbation}), it follows that

\begin{align}\label{relation: balance of w^p and Reynold stress}
	&~{\rm div}\left(\widetilde\omega_{q+1}^{(p)}\otimes\widetilde\omega_{q+1}^{(p)}\right)+{\rm div}\mathring{R}^{(q)}_\ell\nonumber\\
	&=\Theta_{q+1}^2\cdot\left[{\rm div}\left(\omega_{q+1}^{(p)}\otimes\omega_{q+1}^{(p)}\right)+{\rm div}\mathring{R}^{(q)}_\ell\right]+\Big(1-\Theta_{q+1}^2\Big){\rm div}\mathring{R}^{(q)}_\ell\nonumber\\
	&=\Theta_{q+1}^2\cdot\left[{\rm div}\sum_{k\in\Lambda}\left(a_k^{(q+1)}\right)^2\mathbb W_k^{(q+1)}\big(\sigma_{q+1}\cdot\big)\mathring\otimes\mathbb W_k^{(q+1)}\big(\sigma_{q+1}\cdot\big)+{\rm div}\mathring{R}^{(q)}_\ell\right]\nonumber\\
	&\ \ \ \ \ \ +\Big(1-\Theta_{q+1}^2\Big){\rm div}\mathring{R}^{(q)}_\ell+{\rm div}\mathring R_{far}^{(q+1)}+\nabla\left(\frac1d\Theta_{q+1}^2\left|\omega_{q+1}^{(p)}\right|^2\right)\nonumber\\
	&=\Theta_{q+1}^2{\rm div}\sum_{k\in\Lambda}\left(a_k^{(q+1)}\right)^2\left(\mathbb W_k^{(q+1)}\big(\sigma_{q+1}\cdot\big)\mathring\otimes\mathbb W_k^{(q+1)}\big(\sigma_{q+1}\cdot\big)-\int_{\mathbb T^d}\mathbb W_k^{(q+1)}\big(x\big)\otimes\mathbb W_k^{(q+1)}\big(x\big){\rm d}x\right)\nonumber\\
	&\ \ \ \ \ \ +\Theta_{q+1}^2\cdot\Big(1-g_{\kappa}^2\left(\varsigma_{q+1}t\right)\Big){\rm div}\mathring{R}^{(q)}_\ell+\Big(1-\Theta_{q+1}^2\Big){\rm div}\mathring{R}^{(q)}_\ell+{\rm div}\mathring R_{far}^{(q+1)}\nonumber\\
	&\ \ \ \ \ \ +\nabla\left[\Theta_{q+1}^2\left(\frac1d\left|\omega_{q+1}^{(p)}\right|^2+g_{\kappa}\left(\varsigma_{q+1}t\right)^2\cdot\varrho_{q+1}\right)\right]\nonumber\\
	&=\nabla p^{(p)}_{q+1}+\Theta_{q+1}^2\cdot\Big(1-g_{\kappa}^2\left(\varsigma_{q+1}t\right)\Big){\rm div}\mathring{R}^{(q)}_\ell+\Big(1-\Theta_{q+1}^2\Big){\rm div}\mathring{R}^{(q)}_\ell+{\rm div}\left(\mathring R_{far}^{(q+1)}+\mathring R_{osc, x}^{(q+1)}\right)~,
\end{align}
where
\begin{align}
	p^{(p)}_{q+1}
	&:={\Theta_{q+1}^2\left[g_{\kappa}\left(\varsigma_{q+1}t\right)^2\varrho_{q+1}+\frac1d\left|\omega_{q+1}^{(p)}\right|^2-\frac1d\sum_{k\in\Lambda}\left|a_k^{(q+1)}\right|^2\int_{\mathbb T^d}\left|\mathbb W_k^{(q+1)}\right|^2dx\right]}~;\nonumber\\
	\label{definition of R_far}
	\mathring R_{far}^{(q+1)}
	&:={\Theta_{q+1}^2}\sum_{k\neq k'}a_k^{(q+1)}a_{k'}^{(q+1)}\mathbb W_k^{(q+1)}\left(\sigma_{q+1}\cdot\right)\mathring\otimes \mathbb W_{k'}^{(q+1)}\left(\sigma_{q+1}\cdot\right)~;\\
	\label{definition of R_osc,x}
	\mathring R_{osc, x}^{(q+1)}
	&:={\Theta_{q+1}^2}\sum_{k\in\Lambda}\mathcal B\left(\nabla\left(a_k^{(q+1)}\right)^2,\mathbb W_k^{(q+1)}\left(\sigma_{q+1}\cdot\right)\mathring\otimes\mathbb W_k^{(q+1)}\left(\sigma_{q+1}\cdot\right)-\int_{\mathbb T^d}\mathbb W_k^{(q+1)}\big(x\big)\mathring\otimes\mathbb W_k^{(q+1)}\big(x\big){\rm d}x\right).
	\end{align}
Here $\mathcal B$ denotes the bilinear anti-divergence operator defined in Appendix \ref{Appendix B}.\

Moreover, using (\ref{relation: balance of w^p and Reynold stress}), (\ref{definition of h_kappa}), (\ref{definition of temporal perturbation}), one can see that
\begin{align}\label{relation: balance of w^p, w^t and Reynold stress}
	&~{\partial_t\widetilde\omega_{q+1}^{(t)}+{\rm div}\left(\widetilde\omega_{q+1}^{(p)}\otimes\widetilde\omega_{q+1}^{(p)}\right)}+{\rm div}\mathring{R}^{(q)}_\ell\nonumber\\
	&=\nabla p^{(p,t)}_{q+1}+{\rm div}\left(\mathring R_{far}^{(q+1)}+\mathring R_{osc, x}^{(q+1)}+\mathring R_{osc,t}^{(q+1)}\right){+\Big(1-\Theta_{q+1}^2\Big){\rm div}\mathring{R}^{(q)}_\ell+2\Theta_{q+1}\Theta_{q+1}'\omega_{q+1}^{(t)}}~,
\end{align}
with
\begin{align}
	p^{(p,t)}_{q+1}
	&:=p^{(p)}_{q+1}-\varsigma_{q+1}^{-1}\Delta^{-1}{\rm div}{\rm div}\partial_t\left({\Theta_{q+1}^2}h_{\kappa}(\varsigma_{q+1}t)\mathring{R}^{(q)}_\ell\right)~,\nonumber\\
	\label{definition for R_osc,t}
	\mathring R_{osc,t}^{(q+1)}
	&:=\varsigma_{q+1}^{-1}\cdot{\Theta_{q+1}^2}h_{\kappa}(\varsigma_{q+1}t)\partial_t\mathring{R}^{(q)}_\ell~.
\end{align}\

Then by (\ref{mollified Reynold system}) and (\ref{relation: balance of w^p, w^t and Reynold stress}), we have that $v^{(q+1)}$ solves the random Reynold system
\begin{equation}
	\left\{\begin{aligned}
		&\partial_t v^{(q+1)}+{\rm div}\left[\left(v^{(q+1)}+z\right)\otimes \left(v^{(q+1)}+z\right)\right]=-\nabla p_{q+1}+\nu\Delta v^{(q+1)}+z+{\rm div}\mathring{R}^{(q+1)}\\
		&{\rm div}~v^{(q+1)}\equiv 0\\
		&v^{(q+1)}(0)=0
	\end{aligned}\right.\nonumber
\end{equation}
on $[0,\infty)\times\mathbb T^d$, with some gradient pressure $p_{q+1}$ and the new Reynold stress
\begin{equation} \label{construction for R^(q+1)}
	\mathring{R}^{(q+1)}:=\mathring{R}_{com}^\ell+\mathring R_{com}^{(q+1)}+\mathring R_{far}^{(q+1)}+\mathring R_{osc, x}^{(q+1)}+\mathring R_{osc,t}^{(q+1)}+\mathring R_{lin}^{(q+1)}+\mathring R_{cor}^{(q+1)}{+\mathring R_{cut}^{(q+1)}}
\end{equation}
given  by (\ref{definition of molli-communication Reynold stress}), (\ref{definition of R_far}), (\ref{definition of R_osc,x}), (\ref{definition for R_osc,t}), and
\begin{align} \label{definition of R_lin}
	&\mathring R_{lin}^{(q+1)}:=\mathcal{R}\left[\Theta_{q+1}\partial_t\left(\omega_{q+1}^{(p)}+\omega_{q+1}^{(c)}\right)-\nu\Delta\omega_{q+1}+(z_\ell-z)\right]+\left(v^{(q)}_\ell+z_\ell\right)\mathring\otimes\omega_{q+1}+\omega_{q+1}\mathring\otimes\left(v^{(q)}_\ell+z_\ell\right),\\
	\label{definition of R_cor}
	&\mathring R_{cor}^{(q+1)}:=\left({\widetilde\omega_{q+1}^{(c)}+\widetilde\omega_{q+1}^{(t)}}\right)\mathring\otimes~\omega_{q+1}+\widetilde\omega_{q+1}^{(p)}\mathring\otimes\left({\widetilde\omega_{q+1}^{(c)}+\widetilde\omega_{q+1}^{(t)}}\right),\\
	\label{definition of R_com}
	&\mathring R_{com}^{(q+1)}:=\left(v^{(q+1)}+z_\ell\right)\mathring\otimes\left(z-z_\ell\right)+\left(z-z_\ell\right)\mathring\otimes\left(v^{(q+1)}+z_\ell\right)+\left(z-z_\ell\right)\mathring\otimes\left(z-z_\ell\right)~,\\
	\label{definition of R_cut}
	&{\mathring R_{cut}^{(q+1)}:=\Big(1-\Theta_{q+1}^2\Big)\mathring{R}^{(q)}_\ell+\Theta_{q+1}'\mathcal{R}\Big(\omega_{q+1}^{(p)}+\omega_{q+1}^{(c)}+2\Theta_{q+1}\omega_{q+1}^{(t)}\Big)}~.
\end{align}
Here $\mathcal{R}$ is the anti-divergence we recall in Appendix \ref{Appendix B}.

\section{Proof of Proposition \ref{prop. main iteration}} \label{sec. Inductive Estimates}\
\par This section is devoted to the proof for the Main Iteration Proposition \ref{prop. main iteration}. In the following $\alpha,r,p$ are fixed and given in the statement of Proposition \ref{prop. main iteration}. First of all, we start the proof by fixing the parameters in Section \ref{subsec. setting of iteration parameters}. In Section \ref{subsec. Moments Estimates} we prove two propositions. One gives the finite moment bound for the approximate velocity and the error at step $q$ is independent of parameters at step $q+1$, which is the key point for the whole proof. The other shows the mollification convergence which is required for various subsequent estimates. Section \ref{subsec. iterated estimates} is the main part of the proof and contains inductive estimates of approximate velocity and error.

\subsection{Choice of Parameters }\label{subsec. setting of iteration parameters}
\indent\par In the sequel the mollification, concentration and oscillation parameters : $\ell_{q+1}$, $\mu_{q+1}$, $\sigma_{q+1}$, $\kappa_{q+1}$ and $\varsigma_{q+1}$ have to be carefully chosen in order to respect all the conditions appearing in the estimates below. To this end, we  first choose an universal and sufficiently small constant $0<\vartheta<\frac{1}{2d+9}$ such that
$$(d+3)\vartheta\leq \min\left\{2\left(\frac1p-\frac12\right),\frac{d-1}{4r}\right\},$$
which leads to
\begin{equation}
	\left\{\begin{aligned} \label{setting of vartheta}
		&\frac{1}{2\vartheta}\geq(4d+7)\vartheta+\frac{d-1}{2}~,\\
		&\left(\frac{1}{p}-\frac{1}{2}\right)\Big(1+d-(5d+17)\vartheta\Big)\geq\dfrac{d+3}{2}\vartheta ~,\\
		&\left(\frac{1}{p}-\frac{1}{2}\right)\frac{1}{\vartheta}\geq\frac{d-1}{2} ~,\\
		&\frac{d-1}{r}\geq(4d+12)\vartheta~.\\
	\end{aligned}\right.
\end{equation}\\
Now we choose the parameters using $\vartheta$ as follows:\\
\begin{equation}\label{setting of iteration parameters}
\begin{aligned}
	Mollification&:~~~\ell_{q+1}=\lambda_{q+1}^{-\vartheta}~;\\
	Spacial~Concentration&:~~~\mu_{q+1}=\lambda_{q+1}~;\\
	Spacial~Oscillation&:~~~\sigma_{q+1}=\left\lceil\lambda_{q+1}^{\frac{1}{2\vartheta}}\right\rceil~;\\
	Temporal~Concentration&:~~~\kappa_{q+1}=\lambda_{q+1}^{1+d-(5d+17)\vartheta+\frac{1}{\vartheta}}~;\\
	Temporal~Oscillation&:~~~\varsigma_{q+1}=\left\lceil\lambda_{q+1}^{(d+6)\vartheta}\right\rceil~;
\end{aligned}
\end{equation}\\
where "$\lceil x\rceil$" means the smallest integer larger than $x\in\mathbb R$~. As one will see in the final subsection, the final control of $\varpi_{q+1}$ and $\mathring{R}^{(q+1)}$~ is small by choosing $\lambda_{q+1}$ large enough.

\subsection{Preparations}\label{subsec. Moments Estimates}
\indent\par In this subsection we first show that the moment estimates of $v^{(q)}$ and $\mathring{R}^{(q)}$ is finite and independent of $\lambda_{q+1}$, i.e. parameters at level $q+1$. Since the equation has quadratic nonlinearity, the estimates of moments at step $q+1$ contain higher moments of step $q$. By using the following proposition the higher moments only depends on parameters up to step $q$ and we could choose parameters at step $q+1$ to guarantee smallness in the proof.

\begin{proposition}\label{prop. Finiteness of All Moments of Supremum Norms}
	{\bf (Finiteness of All Moments)} For each $q\in\mathbb N_0$ and  any $0<m<\infty$ and $N\in\mathbb N_0$, there exists a finite constant $C_{m,N,\lambda_1,...,\lambda_q}>0$ independent of $\lambda_{q+1}$ such that
	\begin{equation} \label{finiteness of all moments of v^q and R^q}
		\sup_{s\geq0}\mathbb E\left\|v^{(q)}\right\|_{N,s}^m+\sup_{s\geq0}\mathbb E\left\|\mathring{R}^{(q)}\right\|_{C_{s}L^1}^m\leq C_{m,N,\lambda_1,...,\lambda_q}~.
	\end{equation}
\end{proposition}\

\par To prove Proposition \ref{prop. Finiteness of All Moments of Supremum Norms}, we need the following $C^N$-estimates for the coefficients $a_k^{(q+1)}$ ($k\in\Lambda$), which will be used also in section \ref{subsec. iterated estimates}.\

\begin{lemma}\label{lem. subtle estimates for a_k}
	For each $m,N\in\mathbb N_0$ and $k\in\Lambda$, there exists a sequence $\left\lbrace C_{N}\right\rbrace$ of deterministic constants such that for all $s\geq0$ and $t\in[s,s+1]$ we have
\begin{align}\label{subtle estimate for partial a_k}
	\left\|\partial_t^m a_k^{(q+1)}(t)\right\|_{N}\lesssim4^{(m+N)q}\sum_{j=0}^{m}\varsigma_{q+1}^{~j}\big|g^{(j)}_{\kappa}\left(\varsigma_{q+1}t\right)\big|\ell_{q+1}^{-(N+m-j)-\left(N+m+\frac{1}{2}\right)(d+1)}\left(1+\left\|\mathring{R}^{(q)}\right\|_{C_{[s-1,s+1]}L^1}^{N+m+3/2}\right),
\end{align}
Here the implicit constant is deterministic and independent of $\lambda_{q+1}$ .
\end{lemma}\

\begin{proof}[Proof of Lemma \ref{lem. subtle estimates for a_k}]
By the Sobolev embedding $W^{d+1,1}\hookrightarrow C$  and  mollification estimates, we obtain
\begin{align}\label{mol}
	\left\|\mathring{R}^{(q)}_\ell\right\|_{N,s}\lesssim \ell^{-(d+1)-N}\left\|\mathring{R}^{(q)}\right\|_{{C_{[s-1,s+1]}L^1}}.
\end{align}
By Leibniz rule, we get
	\begin{align}
		\left[\partial_t^m a_k^{(q+1)}(t)\right]_N
		\lesssim\sum_{j=0}^m\varsigma_{q+1}^{~j}\big|g^{(j)}_{\kappa}\left(\varsigma_{q+1}t\right)\big|\cdot\left[\varrho_{q+1}^{1/2}\cdot\Gamma_k\left({\rm Id}-\frac{\mathring{R}^{(q)}_\ell}{\varrho_{q+1}}\right)\right]_{N+m-j,s}.\nonumber
	\end{align}
Then it suffices to show for all $N\in\mathbb N_0$ that
	\begin{align}\label{uniform estimate for varrho-Gamma}
		\left[\varrho_{q+1}^{1/2}\cdot\Gamma_k\left({\rm Id}-\frac{\mathring{R}^{(q)}_\ell}{\varrho_{q+1}}\right)\right]_{N,s}\lesssim 4^{Nq} \cdot\ell_{q+1}^{-N-{\left(N+\frac{1}{2}\right)}(d+1)}\left(1+\left\|\mathring{R}^{(q)}\right\|_{C_{[s-1,s+1]}L^1}^{N+3/2}\right).
	\end{align}
Using Leibniz rule again, we obtain
	\begin{align}
		\left[\varrho_{q+1}^{1/2}\cdot\Gamma_k\left({\rm Id}-\frac{\mathring{R}^{(q)}_\ell}{\varrho_{q+1}}\right)\right]_{N,s}\lesssim\sum_{j=0}^N\left[\varrho_{q+1}^{1/2}\right]_{j,s}\left[\Gamma_k\left({\rm Id}-\frac{\mathring{R}^{(q)}_\ell}{\varrho_{q+1}}\right)\right]_{N-j,s}.\nonumber
	\end{align}
In the following we prove for all $N\in\mathbb N_0$ that
\begin{align}\label{uniform estimate for varrho^(1/2)}
	\left[\varrho_{q+1}^{1/2}\right]_{N,s}\lesssim\left\{
		\begin{aligned}
			&{\ell_{q+1}^{-\frac{d+1}{2}}\left(1+\left\|\mathring{R}^{(q)}\right\|_{C_{[s-1,s+1]}L^1}^{1/2}\right)}~~~~~~~~~~~~\quad N=0\\
			&\\
			&4^{Nq}\cdot\ell_{q+1}^{-N-N(d+1)}\left(1+\left\|\mathring{R}^{(q)}\right\|_{C_{[s-1,s+1]}L^1}^N\right) \quad N\geq1
		\end{aligned}\right.
\end{align}
and
\begin{align}\label{uniform estimate for Gamma_k}
	\left[\Gamma_k\left({\rm Id}-\frac{\mathring{R}^{(q)}_\ell}{\varrho_{q+1}}\right)\right]_{N,s}\lesssim
			4^{Nq}\cdot\ell_{q+1}^{-N-(d+1)N}\left(1+\left\|\mathring{R}^{(q)}\right\|_{C_{[s-1,s+1]}L^1}^{N+1}\right),
\end{align}
which implies \eqref{uniform estimate for varrho-Gamma} and the final result. In the following, we write $[\cdot]_N$" instead of " $[\cdot]_{N,s}$" if there's no confusion.
\par First, note that the case $N=0$ of (\ref{uniform estimate for varrho^(1/2)}) and (\ref{uniform estimate for Gamma_k}) is immediate by  (\ref{definition for rescaling function varrho})  and the definition of $\chi_{q+1}$ and $\Gamma_k$. So we only show the case $N\geq1$. This is achieved by using the $C^N$-estimates (\ref{C^N-estimate for compositions 1}) and (\ref{C^N-estimates for compositions 2}) for compositions given in \cite{BLIL15} which we recall in Appendix \ref{Appendix C}.
\par By (\ref{C^N-estimate for compositions 1}),  (\ref{definition for rescaling function varrho}) and \eqref{mol}, we have
	\begin{align}\label{uniform estimate for varrho}
		\big[\varrho_{q+1}\big]_N
		&\lesssim \big[\chi_{q+1}\big]_1\left[\mathring{R}^{(q)}_\ell\right]_N+\big\|{\rm D}\chi_{q+1}\big\|_{N-1}\left\|\mathring{R}^{(q)}_\ell\right\|_0^{N-1}\left[\mathring{R}^{(q)}_\ell\right]_N\nonumber\\
		&\lesssim\ell_{q+1}^{-N-(d+1)N}\left(1+\left\|\mathring{R}^{(q)}\right\|_{C_{[s-1,s+1]}L^1}^{N}\right)~,~~~N\geq1.
	\end{align}
	Then, we apply (\ref{C^N-estimates for compositions 2}) to the function $\Psi_1(y)=y^{1/2}~~\big(~y\in(4^{-q}/2,+\infty)~\big)$ and by (\ref{uniform estimate for varrho}) we obtain that
	\begin{align}
		\left[\varrho_{q+1}^{1/2}\right]_{N}&\lesssim\big[\Psi_1\big]_1\big[\varrho_{q+1}\big]_N+\big\|{\rm D}\Psi_1\big\|_{N-1}\big[\varrho_{q+1}\big]_1^N\nonumber\\
		&\lesssim\left(2^q\big\|{\rm D}\chi_{q+1}\big\|_{N-1}+4^{(N-1/2)q}\big\|{\rm D}\chi_{q+1}\big\|_{0}^N\right)\cdot\ell_{q+1}^{-N-(d+1)N}\left(1+\left\|\mathring{R}^{(q)}\right\|_{C_{[s-1,s+1]}L^1}^{N}\right)\nonumber\\
		&\lesssim4^{Nq}\ell_{q+1}^{-N-(d+1)N}\left(1+\left\|\mathring{R}^{(q)}\right\|_{C_{[s-1,s+1]}L^1}^{N}\right)~,~~~N\geq1,\nonumber
	\end{align}
which implies (\ref{uniform estimate for varrho^(1/2)}). Similarly, we apply (\ref{C^N-estimates for compositions 2}) to the function $\Psi_2(y)=y^{-1}~~\big(~y\in(4^{-q}/2,+\infty)~\big)$  and use (\ref{uniform estimate for varrho}) to deduce that
	\begin{align}\label{uniform estimate for varrho^(-1)}
		\left[\varrho_{q+1}^{-1}\right]_{N}
		&\lesssim\big[\Psi_2\big]_1\big[\varrho_{q+1}\big]_N+\big\|{\rm D}\Psi_2\big\|_{N-1}\big[\varrho_{q+1}\big]_1^N\nonumber\\
		&\lesssim4^{(N+1)q}\ell_{q+1}^{-N-(d+1)N}\left(1+\left\|\mathring{R}^{(q)}\right\|_{C_{[s-1,s+1]}L^1}^{N}\right)~,~~~N\geq1.
	\end{align}
Now we proceed with a bound for $\Gamma_k\left({\rm Id}-\frac{\mathring{R}^{(q)}_\ell}{\varrho_{q+1}}\right)$. By   (\ref{C^N-estimate for compositions 1}), we have to estimate the following:
\begin{align*}
	\left[\frac{\mathring{R}^{(q)}_\ell}{\varrho_{q+1}}\right]_N+\left\|\frac{\mathring{R}^{(q)}_\ell}{\varrho_{q+1}}\right\|_0^{N-1}\left[\frac{\mathring{R}^{(q)}_\ell}{\varrho_{q+1}}\right]_N\lesssim 	\left[\frac{\mathring{R}^{(q)}_\ell}{\varrho_{q+1}}\right]_N.
\end{align*}
	 Then by \eqref{mol}, (\ref{uniform estimate for varrho^(-1)}) and (\ref{definition for rescaling function varrho}), we have
	\begin{align*}
		\left[\frac{\mathring{R}^{(q)}_\ell}{\varrho_{q+1}}\right]_N
		&\lesssim\sum_{m=1}^N\left[\varrho_{q+1}^{-1}\right]_{m}\left[\mathring{R}^{(q)}_\ell\right]_{N-m}+\left[\varrho_{q+1}^{-1}\right]_{0}\left[\mathring{R}^{(q)}_\ell\right]_{N}\nonumber\\
		&\lesssim\sum_{m=1}^N\left[4^{(m+1)q}\cdot\ell_{q+1}^{-N-N(d+1)}\left(1+\left\|\mathring{R}^{(q)}\right\|_{C_{[s-1,s+1]}L^1}^{m+1}\right)\right.\nonumber\\
		&\ \ \ \ \ \ \ \ \ \ \ \ \left.+4^{(m+1)q}\cdot\ell_{q+1}^{-N-N(d+1)}\left(1+\left\|\mathring{R}^{(q)}\right\|_{C_{[s-1,s+1]}L^1}\right)\right]\nonumber\\
		&\lesssim4^{(N+1)q}\cdot\ell_{q+1}^{-N-N(d+1)}\left(1+\left\|\mathring{R}^{(q)}_\ell\right\|_{C_{[s-1,s+1]}L^1}^{N+1}\right)~,~~~N\geq1.
	\end{align*}
Thus we obtain \eqref{uniform estimate for Gamma_k} and the proof is complete. 	
\end{proof}\

\begin{proof}[Proof for Proposition \ref{prop. Finiteness of All Moments of Supremum Norms}.] By the construction we see that $v^{(q+1)}$ and $\mathring{R}^{(q+1)}$ only depends on  $\lambda_1,...,\lambda_{q+1}$.  Hence, we only need to show the right hand side of (\ref{finiteness of all moments of v^q and R^q}) is finite.
\par We prove the estimate by induction. For $q=0$~, (\ref{finiteness of all moments of v^q and R^q}) follows directly from  (\ref{start point of the iteration}) and (\ref{z finiteness of moments}). Now assume (\ref{finiteness of all moments of v^q and R^q}) holds for some $q\in\mathbb N_0$ by induction. We are then prove \eqref{finiteness of all moments of v^q and R^q} for $q+1$. By (\ref{subtle estimate for partial a_k}) we have
\begin{align}\label{uniform estimate for partial a_k}
	&\left\|a_k^{(q+1)}\right\|_{N,s}\lesssim1+\left\|\mathring{R}^{(q)}\right\|_{C_{[s-1,s+1]}L^1}^{N+3/2}~,~~~\forall~s\geq0 ;~ N\in\mathbb N_0,~k\in\Lambda~.
\end{align}
Here, the implicit finite constant is deterministic and only depends on $N$, $q$ and $\lambda_{q+1}$.
\par  We start with the estimate of perturbations.
\par\noindent$\bullet$ {\bf Estimates of perturbations :}

By (\ref{definition of principal pertutbation}), (\ref{definition of corrector perturbation}), (\ref{estimates for Mikado flows W}), (\ref{estimates for Mikado flows V}) and (\ref{uniform estimate for partial a_k}), we have for all $s\geq0$ that
\begin{align}\label{power estimate of w^(p)}
	\left\|\omega^{(p)}_{q+1}\right\|_{N,s}
	&\lesssim\sum_{k\in\Lambda}\left\|a_k^{(q+1)}\right\|_{N,s}\left\|\mathbb W^{(q+1)}_k\right\|_{W^{N,\infty}}\nonumber\\
	&\lesssim1+\left\|\mathring{R}^{(q)}\right\|_{C_{[s-1,s+1]}L^1}^{N+3/2}~;
\end{align}
\begin{align}\label{power estimate of w^(c)}
	\left\|\omega^{(c)}_{q+1}\right\|_{N,s}
	&\lesssim\sum_{k\in\Lambda}\left\|\nabla a_k^{(q+1)}\right\|_{N,s}\left\|\mathbb V^{(q+1)}_k\right\|_{W^{N,\infty}}~\nonumber\\
	&\lesssim1+\left\|\mathring{R}^{(q)}\right\|_{C_{[s-1,s+1]}L^1}^{N+5/2}~;
\end{align}
and by (\ref{definition of temporal perturbation}),  mollification estimate and Sobolev embedding, we have for all $s\geq0$ that
\begin{align}\label{power estimate of w^(t)}
	\left\|\omega^{(t)}_{q+1}\right\|_{N,s}
	&\lesssim\left\|\mathring{R}^{(q)}\right\|_{C_{[s-1,s+1]}L^1}.
\end{align}
For $\tilde \omega^{(p)}_{q+1},\tilde \omega^{(c)}_{q+1},\tilde \omega^{(t)}_{q+1} $ we see that the $n$-th derivative of $\Theta_{q+1}$ behaves like $\ell_{q+1}^{-n/2}$ does not pose any problems as the $C^N$ norm of $\omega^{(p)}_{q+1}, \omega^{(c)}_{q+1}, \omega^{(t)}_{q+1}$ has more powers of $\ell_{q+1}^{-1}$.
Taking $m$th-moment and then supremum for $s\in[0,\infty)$ in the above estimates, we have for $m\in\mathbb N$ that
\begin{align}\label{moment estimates of v^(q+1)}
	\sup_{s\geq0}\mathbb E\left\|v^{(q+1)}\right\|_{N,s}^m
	&\lesssim\sup_{s\geq0}\mathbb E\left\|v^{(q)}_{\ell_{q+1}}\right\|_{N,s}^m+\sup_{s\geq0}\mathbb E\big\|\omega_{q+1}\big\|_{N,s}^m\nonumber\\
	&\lesssim\sup_{s\geq0}\mathbb E\left\|v^{(q)}\right\|_{N,s}^m+1+\sup_{s\geq0}\mathbb E\left\|\mathring{R}^{(q)}\right\|_{C_{s}L^1}^{\left(N+\frac{5}{2}\right)m}<\infty~.
\end{align}
\par\noindent$\bullet$ {\bf Estimates of Reynold stress errors :}

By (\ref{definition of molli-communication Reynold stress}) and  mollification estimate, we have for all $s\geq0$ that
\begin{align}
	\left\|\mathring{R}_{com}^{\ell_{q+1}}\right\|_{C_{s}L^1}
	&\lesssim\left\|v^{(q)}\right\|_{0,[s-1,s+1]}^2+\left\|z\right\|_{C_{[s-1,s+1]}L^2}^2~.\nonumber
\end{align}
Moreover, by (\ref{definition of R_far}), (\ref{definition of R_osc,x}), {(\ref{def. cutoff Theta}),} (\ref{L^p estimate for bilinear reverse divergence}), (\ref{estimates for Mikado flows W}) and (\ref{uniform estimate for partial a_k}), we have for all $s\geq0$ that
\begin{align}
	\left\|\mathring{R}_{far}^{(q+1)}\right\|_{C_{s}L^1}
	&\lesssim{\big\|\Theta_{q+1}\big\|_{0}^2}\sum_{k\neq k'}\left\|a_k^{(q+1)}\right\|_{0,s}\left\|a_{k'}^{(q+1)}\right\|_{0,s}\left\|\mathbb W^{(q+1)}_k\right\|_{L^2}\left\|\mathbb W^{(q+1)}_{k'}\right\|_{L^2}\nonumber\\
	&\lesssim1+\left\|\mathring{R}^{(q)}\right\|^{3}_{C_{[s-1,s+1]}L^1}~,\nonumber\\
	\left\|\mathring{R}_{osc,x}^{(q+1)}\right\|_{C_{s}L^1}
	&\lesssim{\big\|\Theta_{q+1}\big\|_{0}^2}\sum_{k\in\Lambda}\left\|\nabla\left|a_k^{(q+1)}\right|^2\right\|_{C_{s}C^1}\left\|\mathbb W_k^{(q+1)}\mathring\otimes \mathbb W_k^{(q+1)}\right\|_{L^1}\nonumber\\
	&\lesssim\sum_{k\in\Lambda}\left\|a_k^{(q+1)}\right\|_{1,s}\left\|a_k^{(q+1)}\right\|_{2,s}\left\|\mathbb W_k^{(q+1)}\right\|_{L^2}^2\nonumber\\
	&\lesssim1+\left\|\mathring{R}^{(q)}\right\|_{C_{[s-1,s+1]}L^1}^{6}~.\nonumber
\end{align}
By (\ref{definition for R_osc,t}), {(\ref{def. cutoff Theta})} and  mollification estimate, we have for all $s\geq0$ that
\begin{align}
	\left\|\mathring{R}_{osc,t}^{(q+1)}\right\|_{C_{s}L^1}
	&\lesssim\left\|\mathring{R}^{(q)}\right\|_{C_{[s-1,s+1]}L^1}~,\nonumber
\end{align}
and by (\ref{definition of R_lin}), {(\ref{def. cutoff Theta}),} ${L^2}$-boundedness of the operators $\mathcal R$ and $\mathcal R{\rm div}$,
 (\ref{definition of R_cor}), (\ref{power estimate of w^(p)}), (\ref{power estimate of w^(c)}) and (\ref{power estimate of w^(t)}), we have for all $s\in[0,\infty)$ that
\begin{align}
	\left\|\mathring{R}_{lin}^{(q+1)}\right\|_{C_{s}L^1}
	&\lesssim{\big\|\Theta_{q+1}\big\|_{0}}\left\|\omega^{(p)}_{q+1}+\omega^{(c)}_{q+1}\right\|_{1,s}+\big\|\omega_{q+1}\big\|_{1,s}+\left\|z\right\|_{C_{[s-1,s+1]}L^2}\nonumber\\
	&~~~~~+\big\|\omega_{q+1}\big\|_{C_{s}L^2}\left(\left\|v^{(q)}\right\|_{C_{[s-1,s+1]}L^2}+\left\|z\right\|_{C_{[s-1,s+1]}L^2}\right)\nonumber\\
	&\lesssim1+\left\|\mathring{R}^{(q)}\right\|_{C_{[s-1,s+1]}L^1}^{5}+\left\|v^{(q)}\right\|_{0,[s-1,s+1]}^2+\left\|z\right\|_{C_{[s-1,s+1]}L^2}^2~,\nonumber\\
	\left\|\mathring{R}_{cor}^{(q+1)}\right\|_{C_{s}L^1}
	&\lesssim{\big\|\Theta_{q+1}\big\|_{0}^2}\left\|\omega^{(p)}_{q+1}\right\|_{0,s}^2+{\big\|\Theta_{q+1}\big\|_{0}^2}\left\|\omega^{(c)}_{q+1}\right\|_{0,s}^2+{\big\|\Theta_{q+1}\big\|_{0}^4}\left\|\omega^{(t)}_{q+1}\right\|_{0,s}^2\nonumber\\
	&\lesssim1+\left\|\mathring{R}^{(q)}\right\|_{C_{[s-1,s+1]}L^1}^{5}~.\nonumber
\end{align}
Finally by (\ref{definition of R_com}), and  mollification estimate, we have for all $s\geq0$ that
\begin{align}
	\left\|\mathring{R}_{com}^{(q+1)}\right\|_{C_{s}L^1}
	&\lesssim\left\|v^{(q+1)}\right\|_{0,s}^2+\left\|z\right\|_{C_{[s-1,s+1]}L^2}^2\nonumber\\
	&\lesssim1+\left\|\mathring{R}^{(q)}\right\|_{C_{[s-1,s+1]}L^1}^{5}+\left\|v^{(q)}\right\|_{0,s}^2+\left\|z\right\|_{C_{[s-1,s+1]}L^2}^2~;~~~\nonumber
\end{align}
{and by \ref{definition of R_cut}), (\ref{def. cutoff Theta}), standard mollification estimate, $L^1$-boundedness of the operator $\mathcal R$, (\ref{power estimate of w^(p)}), (\ref{power estimate of w^(c)}) and (\ref{power estimate of w^(t)}), we have for all $s\in[0,\infty)$ that
\begin{align}
	\left\|\mathring{R}_{cut}^{(q+1)}\right\|_{C_{s}L^1}
	&\lesssim\left\|\mathring{R}_{\ell}^{(q)}\right\|_{C_{s}L^1}+\left\|\omega^{(p)}_{q+1}\right\|_{0,s}+\left\|\omega^{(c)}_{q+1}\right\|_{0,s}+\left\|\omega^{(t)}_{q+1}\right\|_{0,s}\nonumber\\
	&\lesssim1+\left\|\mathring{R}^{(q)}\right\|_{C_{[s-1,s+1]}L^1}^{5/2}~.\nonumber
\end{align}}
Taking $m$th-moment and then supremum for $s\geq0$ in the above estimates, and together by (\ref{moment estimates of v^(q+1)}), we have for $m\in\mathbb N$ that
\begin{align}
	\sup_{s\geq0}\mathbb E\left\|\mathring{R}^{(q+1)}\right\|_{C_{s}L^1}^m
	&\lesssim1+\sup_{s\geq0}\mathbb E\left\|\mathring{R}^{(q)}\right\|_{C_{s}L^1}^{6m}+\sup_{s\geq0}\mathbb E\left\|v^{(q)}\right\|_{0,s}^{2m}+\sup_{s\geq0}\mathbb E\big\|z\big\|_{C_{s}L^2}^{2m}<\infty~.
\end{align}
This completes the proof.	
\end{proof}\

\par As the next step, we show the following mollification convergence which are used in the estimate of approximate velocity and the Reynold error.\\

\begin{proposition}\label{prop. uniform convergence of mollified z}
Assume $u_0\in L^2_\sigma$ and $\textrm{Tr}(G^*G)<\infty$.
Then, for each $q\in \mathbb N_0$,
\begin{align}\label{uniform convergence of mollified z}
	&\big\|z-z_\ell\big\|_{\bar L^2(\Omega,L^2_sL^2)}
	\longrightarrow0~,\ \ \ as\ \ \ \ell\rightarrow0~,
\end{align}
\begin{align}\label{uniform convergence of mollified v^(q)}
	&\left\|v^{(q)}-v^{(q)}_\ell\right\|_{\bar L^2(\Omega,L^2_sL^2)}+\left\|v^{(q)}-v^{(q)}_\ell\right\|_{ L^\alpha(\Omega,\mathbb{Z}^{p,r})}+\left\|v^{(q)}-v^{(q)}_\ell\right\|_{\mathbf{E}_p}\lesssim C_q\ell~,
\end{align}
\begin{align}\label{uniform convergence of mollified v^(q)Oz}
	&\left\|v^{(q)}\mathring\otimes z-\left(v^{(q)}\mathring\otimes z\right)_\ell\right\|_{\bar L^1(\Omega,L^1_sL^1)}
	\longrightarrow0~,\ \ \ as\ \ \ \ell\rightarrow0~.
\end{align}	
Here $C_q$ is a constant depending only on $\lambda_1,...\lambda_{q}$ and $\left(v^{(q)}\mathring\otimes z\right)_\ell:=\left(v^{(q)}\mathring\otimes z\right)\ast_x\phi_\ell\ast_t\varphi_\ell$ .
\end{proposition}\

To prove Proposition \ref{prop. uniform convergence of mollified z}, we give a lemma. To this end, we introduce the following notations for stochastic vector fields $u\in\bar L^2\big(\Omega,L^2_sL^2\big)$ and tensor fields $R\in\bar L^1\big(\Omega,L^1_sL^1\big)$ :
\begin{align}
	\mathbb T_\ell\big(u;s_0\big)&:=\sup_{s\geq s_0}\mathbb E\int^{s+1}_{s}\int_{\mathbb R}\big\|u(t)-u(t-\tau)\big\|_{L^2}^2\varphi_\ell(\tau){\rm d}\tau{\rm d}t~,\nonumber\\
	\mathbb S_\ell\big(u;s_0\big)&:=\sup_{s\geq s_0}\mathbb E\int^{s+1}_{s}\int_{\mathbb R^d}\big\|u(t)-\tau_{-y}u(t)\big\|_{L^2}^2\phi_\ell(y){\rm d}y{\rm d}t~,\nonumber\\
	\mathbb T'_\ell\big(R;s_0\big)&:=\sup_{s\geq s_0}\mathbb E\int^{s+1}_{s}\int_{\mathbb R}\big\|R(t)-R(t-\tau)\big\|_{L^1}\varphi_\ell(\tau){\rm d}\tau{\rm d}t~,\nonumber\\
	\mathbb S'_\ell\big(R;s_0\big)&:=\sup_{s\geq s_0}\mathbb E\int^{s+1}_{s}\int_{\mathbb R^d}\big\|R(t)-\tau_{-y}R(t)\big\|_{L^1}\phi_\ell(y){\rm d}y{\rm d}t~,\nonumber
\end{align}
where  $s_0\in\mathbb{R},\tau_{y}u:=u(\cdot+y)$. For simplicity we write $\mathbb T_\eta\big(u\big):=\mathbb T_\eta\big(u;0\big)$.

\begin{lemma}\label{lem. T,S-estimates}
{We} have the following estimates:
	\begin{align}
		\sup_{s\geq 0}\mathbb E\int^{s+1}_{s}\big\|u(t)-u_\ell(t)\big\|_{L^2}^2{\rm d}t
	&\lesssim\mathbb T_\ell\big(u\big)+\mathbb S_\ell\big(u;-1\big)~,\label{sup-estimate for vector fields}\\
	\sup_{s\geq 0}\mathbb E\int^{s+1}_{s}\big\|R(t)-R_\ell(t)\big\|_{L^1}{\rm d}t
	&\leq\mathbb T'_\ell\big(R\big)+\mathbb S'_\ell\big(R;-1\big)~,\label{sup-estimate for matrices}
	\end{align}

In particular, we have
	\begin{align}\label{sup-estimate for uOv}
		&\big\|v^{(q)}\mathring\otimes z-(v^{(q)}\mathring\otimes z)_\ell\big\|_{\bar L^1(\Omega,L^1_sL^1)}\\\nonumber
		&\lesssim\left(\big\|v^{(q)}\big\|_{\bar L^2(\Omega,L^2_sL^2)}+\big\|z\big\|_{\bar L^2(\Omega,L^2_sL^2)}+\|u_0\|_{L^2}\right)\Big(\mathbb T_\ell\big(v^{(q)}\big)+\mathbb S_\ell\big(v^{(q)}\big)+\mathbb T_\ell\big(z\big)+\mathbb S_\ell\big(z;-1\big)\Big)^{1/2}.
	\end{align}
\end{lemma}\

\begin{proof}[Proof of Lemma \ref{lem. T,S-estimates}.]
A straightforward application of H\"older's inequality and Fubbini's Theorem gives
	\begin{align}
		&~\int_s^{s+1}\big\|u(t)-u_\ell(t)\big\|_{L^2}^2{\rm d}t\nonumber\\
		&\lesssim\int_s^{s+1}\big\|u(t)-(\varphi_\ell\ast_t u)(t)\big\|_{L^2}^2{\rm d}t+\int_{s-1}^{s+1}\big\|u(t)-(\phi_\ell\ast_x u)(t)\big\|_{L^2}^2{\rm d}t\nonumber\\
		&\lesssim\int_s^{s+1}\int_{\mathbb R}\big\|u(t)-u(t-\tau)\big\|_{L^2}^2\varphi_\ell(\tau){\rm d}\tau{\rm d}t+\int_{s-1}^{s+1}\int_{\mathbb R^d}\big\|u(t)-\tau_{-y}u(t)\big\|_{L^2}^2\phi_\ell(y){\rm d}y{\rm d}t~;\nonumber
	\end{align}
which implies (\ref{sup-estimate for vector fields}). Similarly (\ref{sup-estimate for matrices}) holds.

Now we apply \eqref{sup-estimate for matrices} to $v^{(q)}\mathring\otimes z$. Note that for $u,v\in L^2_sL^2$
	\begin{align}
		\big\|(u\mathring\otimes v)(t)-(u\mathring\otimes v)(t-\tau)\big\|_{L^1}
		&\lesssim\big\|u(t)\big\|_{L^2}\big\|v(t)-v(t-\tau)\big\|_{L^2}
		+\big\|v(t-\tau)\big\|_{L^2}\big\|u(t)-u(t-\tau)\big\|_{L^2}~,\nonumber\\
		\big\|(u\mathring\otimes v)(t)-\tau_{-y}(u\mathring\otimes v)(t)\big\|_{L^1}
		&\lesssim\big\|u(t)\big\|_{L^2}\big\|v(t)-\tau_{-y}v(t)\big\|_{L^2}
		+\big\|\tau_{-y}v(t)\big\|_{L^2}\big\|u(t)-\tau_{-y}u(t)\big\|_{L^2}~.\nonumber
	\end{align}
Applying H\"older's inequality we obtain
	\begin{align}
		\mathbb T'_\ell\Big(v^{(q)}\mathring\otimes z\Big)
		&\lesssim\big\|v^{(q)}\big\|_{\bar L^2(\Omega,L^2_sL^2)}\cdot\mathbb T_\ell\big(z\big)^{1/2}+\left(\sup_{s\geq0}\mathbb E\int^{s+1}_s\int_{\mathbb R}\big\|z(t-\tau)\big\|_{L^2}^2\varphi_\ell(\tau){\rm d}\tau{\rm d}t\right)^{1/2}\mathbb T_\ell\big(v^{(q)}\big)^{1/2},\nonumber\\
		\mathbb S'_\ell\Big(v^{(q)}\mathring\otimes z\Big)
		&\lesssim\big\|v^{(q)}\big\|_{\bar L^2(\Omega,L^2_sL^2)}\cdot\mathbb S_\ell\big(z;-1\big)^{1/2}+\left(\sup_{s\geq-1}\mathbb E\int^{s+1}_s\int_{\mathbb R^d}\big\|\tau_{-y}z(t)\big\|_{L^2}^2\phi_\ell(y){\rm d}y{\rm d}t\right)^{1/2}\mathbb S_\ell\big(v^{(q)}\big)^{1/2}.\nonumber
	\end{align}
	Moreover, we have
	\begin{align}
		\sup_{s\geq0}\mathbb E\int^{s+1}_s\int_{\mathbb R}\big\|z(t-\tau)\big\|_{L^2}^2\varphi_\ell(\tau){\rm d}\tau{\rm d}t
		&=\sup_{s\geq0}\int_{\mathbb R}\left(\mathbb E\int^{s-\tau+1}_{s-\tau}\big\|z(t)\big\|_{L^2}^2{\rm d}t\right)\varphi_\ell(\tau){\rm d}\tau\nonumber\\
		&\lesssim\big\|z\big\|_{\bar L^2(\Omega,L^2_sL^2)}^2+\|u_0\|_{L^2}^2~,\nonumber\\
		\sup_{s\geq-1}\mathbb E\int^{s+1}_s\int_{\mathbb R^d}\big\|\tau_{-y}z(t)\big\|_{L^2}^2\phi_\ell(y){\rm d}y{\rm d}t
		&\leq\big\|z\big\|_{\bar L^2(\Omega,L^2_sL^2)}^2+\|u_0\|_{L^2}^2~.\nonumber
	\end{align}
which implies the final result.
\end{proof}\

\par With Lemma \ref{lem. T,S-estimates} in hand, we are ready to prove Proposition \ref{prop. uniform convergence of mollified z}.

\begin{proof}[Proof of Proposition \ref{prop. uniform convergence of mollified z}]
We first note that
	\begin{align}
		\mathbb T_\ell\Big(v^{(q)}\Big)+\mathbb S_\ell\Big(v^{(q)};-1\Big)\lesssim\ell_{q+1}^{2}\cdot\sup_{s\in\mathbb R}\mathbb E\left\|v^{(q)}\right\|_{1,s}^2~,\nonumber
\end{align}
{which} implies \eqref{uniform convergence of mollified v^(q)} holds.
Then by lemma \ref{lem. T,S-estimates} it suffices to prove
\begin{align}\label{T,S-convergence of z}
	\mathbb T_\ell\big(z\big)+\mathbb S_\ell\big(z;-1\big)\longrightarrow0~~~\textrm{as}~~~\ell\rightarrow0~.
\end{align}
For the first term it is easy to see that
	$$\mathbb T_\ell\big(z\big)\lesssim\mathbb T_\ell\big(W_{con}\big)+\mathbb T_\ell\big(z^{in}\big),$$
	and
$$\mathbb T_\ell\big(W_{con}\big)\lesssim \ell_{q+1}^{1-2\delta}\sup_{s\geq0}\mathbb E\|W_{con}\|^2_{C^{1/2-\delta}_{[s-1,s+1]}L^2}.$$
Here $z^{in}=e^{t(\Delta-I)}u_0$.

In the following we first consider  $\mathbb T_\ell\big(z^{in}\big)$.
Using the dominated convergence theorem and the fact that $z^{in}$ is $\mathbb P-a.s.$ uniformly continuous in $L^2$ on $[-2,2]$ we obtain
\begin{align}
	&~~~\sup_{-1\leq s\leq 1}\mathbb E\int^{s+1}_{s}\int_{\mathbb R}\big\|z^{in}(t)-z^{in}(t-\tau)\big\|_{L^2}^2\varphi_\ell(\tau){\rm d}\tau{\rm d}t\nonumber\\
	&\lesssim\mathbb E\int^{2}_{-1}\int_{\mathbb R}\big\|z^{in}(t)-z^{in}(t-\tau)\big\|_{L^2}^2\varphi_\ell(\tau){\rm d}\tau{\rm d}t\longrightarrow0~~~as~~~\ell\rightarrow0~.\nonumber
\end{align}
Moreover, we have
\begin{align*}
	&~~~~\sup_{s\geq 1}\mathbb E\int^{s+1}_{s}\int_{\mathbb R}\big\|z^{in}(t)-z^{in}(t-\tau)\big\|_{L^2}^2\varphi_\ell(\tau){\rm d}\tau{\rm d}t\nonumber\\
	&=\sup_{s\geq1}\int^{s+1}_{s}\int_{\mathbb R}\left\|e^{(t-\tau)(\Delta-I)}\left(e^{\tau(\Delta-I)}u_0-u_0\right)\right\|_{L^2}^2\varphi_\ell(\tau){\rm d}\tau{\rm d}t\nonumber\\
	&\leq\int_{\mathbb R}\left\|e^{\tau(\Delta-I)}u_0-u_0\right\|_{L^2}^2\varphi_\ell(\tau){\rm d}\tau\longrightarrow0~~~as~~~\ell\rightarrow0~.
\end{align*}
Thus $\mathbb T_\ell\big(z^{in}\big)\longrightarrow0$.

\par Now we consider  $\mathbb S_\ell\big(z;-1\big)$ and we also seperate into $z^{in}$ and $W_{con}$ part.
First we have for all $t>0$ and $y\in\mathbb R^d$ that
\begin{align}
	\mathbb E\big\|W_{con}(t)-\tau_{-y}W_{con}(t)\big\|_{L^2}^2
	&=\mathbb E\big\|\big(I-\tau_{-y}\big)W_{con}(t)\big\|_{L^2}^2\nonumber\\
	&=\int^{t}_{0}\big\|\big(I-\tau_{-y}\big)\circ \mathbf{P}e^{(t-r)(\Delta-I)}G\big\|_{L_2}^2{\rm d}r\nonumber\\
	&\leq\big\|\big(I-\tau_{-y}\big)\circ G\big\|_{L_2}^2\cdot\int^{t}_{0}e^{-2(t-r)}{\rm d}r\nonumber\\
	&\leq\frac{1}{2}\big\|\big(I-\tau_{-y}\big)\circ G\big\|_{L_2}^2~.\nonumber
\end{align}
Moreover, we have
$$\left\|\varphi_\ell*z^{in}-z^{in}\right\|_{L^2}\lesssim\left\|\varphi_\ell*u_0-u_0\right\|_{L^2}.$$
Then, by Fubini's Theorem and the dominated convergence theorem, we have
\begin{align}
	\mathbb S_\ell\big(z;-1\big)
	&\lesssim\int_{\mathbb R^d}\big\|\big(I-\tau_{-y}\big)\circ G\big\|_{L_2}^2\phi_\ell(y){\rm d}y+\|\varphi_\ell*u_0-u_0\|_{L^2}\\
	&=\int_{\mathbb R^d}\big\|\big(I-\tau_{-\ell y}\big)\circ G\big\|_{L_2}^2\phi(y){\rm d}y+\left\|\varphi_\ell*u_0-u_0\right\|_{L^2}\longrightarrow0~~as~~\ell\rightarrow0~,\nonumber
\end{align}
since for each $y\in\mathbb R^d$,
\begin{align}
	\big\|\big(I-\tau_{-\ell y}\big)\circ G\big\|_{L_2}^2
	&=\sum_{n\in\mathbb N}\left\|Ge_n-\tau_{-\ell y}Ge_n\right\|_{L^2}^2\longrightarrow0~~as~~\ell\rightarrow0,\nonumber\\
	\big\|\big(I-\tau_{-\ell y}\big)\circ G\big\|_{L_2}^2
	&\leq4\big\|G\big\|_{L_2}^2~~\forall\ell>0~.\nonumber
\end{align}
This completes the proof.
\end{proof}\

\subsection{Proof of  Proposition \ref{prop. main iteration}}\label{subsec. iterated estimates}
\indent\par To conclude the proof of Proposition \ref{prop. main iteration} we shall verify \eqref{smallness of Rey.}-\eqref{weak smallness of pert.}. In the following we use $C_q$ to denote deterministic constant that may depend on $\lambda_1,\cdots,\lambda_q$ , $\varphi$, $\phi$, $\chi_1,\cdots,\chi_{q+1}$ and $\Gamma_k$. Note that $C_q$ is independent of $\lambda_{q+1}$ and the end point $s\geq0$ of time intervals $[s,s+1]$. In the following estimates, $C_q$ may change from line to line.

\par First, we recall the following result proved in \cite[Lemma 2.1]{MS18} (see also \cite[Lemma B.1]{CL22}).
\begin{lemma}[\bf Improved H\"older's Inequality on $\mathbb T^d$] Let $1\leq p\leq\infty$ and $a,f\in C^\infty\left(\mathbb T^d\right)$. Then for any $\sigma\in\mathbb N$,
\begin{align} \label{improved Holder's inequality}
	\Big|~\left\|af\left(\sigma\cdot\right)\right\|_{L^p}-\left\|a\right\|_{L^p}\left\|f\right\|_{L^p}\Big|\lesssim\sigma^{-1/p}\left\|a\right\|_{1}\left\|f\right\|_{L^p}~.
\end{align}	
\end{lemma}\
\par This result is applied to bound $\omega^{(p)}_{q+1}$ in $L^2$. By taking $m=0$ and $m=1$ respectively in Lemma \ref{lem. subtle estimates for a_k}, we have for all $s\geq0$ and $t\in[s,s+1]$ that

\begin{align} \label{subtle estimates for a_k}
	\left\|a_k^{(q+1)}(t)\right\|_{N}
	\leq C_{N,q}\cdot\big|g_{\kappa}\left(\varsigma_{q+1}t\right)\big|\cdot\ell_{q+1}^{-N-\big(N+1/2\big)(d+1)}\left(1+\left\|\mathring{R}^{(q)}\right\|_{C_{[s-1,s+1]}L^1}^{N+3/2}\right),
\end{align}
\begin{align} \label{subtle estimate for partial_t a_k}
    \left\|\partial_t a_k^{(q+1)}(t)\right\|_{N}
	&\leq C_{N,q}\Big(\varsigma_{q+1}\big|g'_{\kappa}\left(\varsigma_{q+1}t\right)\big|+\big|g_{\kappa}\left(\varsigma_{q+1}t\right)\big|\ell_{q+1}^{-1}\Big)\ell_{q+1}^{-N-\big(N+3/2\big)(d+1)}\left(1+\left\|\mathring{R}^{(q)}\right\|_{C_{[s-1,s+1]}L^1}^{N+5/2}\right),
\end{align}
Here the deterministic constant $C_{N,q}>0$ is independent of the choice of $\lambda_{q+1}$ and of $s$.\\

\subsubsection{Inductive Estimate of $v^{(q+1)}$}\

\noindent$\bullet$ {\bf $\bar L^2\big(\Omega,L^2_sL^2\big)$-Estimate :}\

By (\ref{definition of principal pertutbation}),  (\ref{estimates for Mikado flows W}), (\ref{subtle estimates for a_k}) and (\ref{improved Holder's inequality}), we have for all $t\in[s,s+1]$ ($s\geq0$) that
\begin{align} \label{t-wise L^2 estimate for w^(p)}
	\left\|\omega^{(p)}_{q+1}(t)\right\|_{L^2}
	&\lesssim\sum_{k\in\Lambda}\bigg(\left\|a_k^{(q+1)}(t)\right\|_{L^2}\left\|\mathbb W^{(q+1)}_k\right\|_{L^2}+\sigma_{q+1}^{-1/2}\left\|a_k^{(q+1)}(t)\right\|_{1}\left\|\mathbb W^{(q+1)}_k\right\|_{L^2}\bigg)\nonumber\\
	&\lesssim\big|g_{\kappa}\left(\varsigma_{q+1}t\right)\big|\bigg[~\Big\|\varrho_{q+1}(t)\Big\|_{L^1}^{1/2}+C_q\cdot\sigma_{q+1}^{-1/2}\ell_{q+1}^{-\frac{3d+5}{2}}\left(1+\left\|\mathring{R}^{(q)}\right\|_{C_{[s-1,s+1]}L^1}^{5/2}\right)\bigg]~.
\end{align}
Notice that
\begin{align}
	\Big\|\varrho_{q+1}(t)\Big\|_{L^1}^{1/2}
	&=2\left(\int_{\mathbb T^d}\chi_{q+1}\left(\mathring{R}^{(q)}_{\ell}(t,x)\right){\rm d}x\right)^{1/2}\nonumber\\
	&\leq C_q\left(1+\left\|\mathring{R}^{(q)}\right\|_{C_{[s-1,s+1]}L^1}^{1/2}\right)~,~~~~~~\forall t\in[s,s+1]\nonumber
\end{align}
and
\begin{align}\label{C^1-norm of ||varrho||_L^1}
	\left\|\Big\|\varrho_{q+1}\Big\|_{L^1_x}^{1/2}\right\|_{C^1[s,s+1]}\leq C_q\cdot\ell_{q+1}^{-1}\left(1+\left\|\mathring{R}^{(q)}\right\|_{C_{[s-1,s+1]}L^1}\right)~.
\end{align}
Taking $L^2[s,s+1]$-norm in (\ref{t-wise L^2 estimate for w^(p)}), and using (\ref{improved Holder's inequality}) again, and by (\ref{L^2 norm for g_kappa}), (\ref{C^1-norm of ||varrho||_L^1}) {and mollification estimate}, we have
\begin{align}\label{only for explanation}
	\left\|\omega^{(p)}_{q+1}\right\|_{L^2_{s}L^2}\lesssim
	&\left\|\mathring{R}^{(q)}\right\|_{L^1_{[s-\ell,s+1]}L^1}^{1/2}+2^{-q}+C_q\left(\varsigma_{q+1}^{-\frac{1}{2}}\ell_{q+1}^{-1}+\sigma_{q+1}^{-\frac{1}{2}}\ell_{q+1}^{-\frac{3d+5}{2}}\right)\left(1+\left\|\mathring{R}^{(q)}\right\|_{C_{[s-1,s+1]}L^1}^{5/2}\right).
\end{align}
Moreover, we obtain
\begin{align}\label{remove bad molli. influence}
	\sup_{s\geq0}\mathbb E\left\|\mathring{R}^{(q)}\right\|_{L^1_{[s-\ell,s+1]}L^1}
	&\leq\sup_{s\geq0}\mathbb E\left\|\mathring{R}^{(q)}\right\|_{L^1_{[s-\ell,s]}L^1}+\left\|\mathring{R}^{(q)}\right\|_{\bar L^1(\Omega,L^1_sL^1)}\nonumber\\
	&\leq\ell_{q+1}\cdot\sup_{s\geq0}\mathbb E\left\|\mathring{R}^{(q)}\right\|_{C_{s}L^1}+\left\|\mathring{R}^{(q)}\right\|_{\bar L^1(\Omega,L^1_sL^1)}~.
\end{align}
Taking the second moment and then supremum for $s\geq0$ {in (\ref{only for explanation})}, by {(\ref{remove bad molli. influence}),} (\ref{finiteness of all moments of v^q and R^q}) and (\ref{setting of iteration parameters}), we have
\begin{align}\label{inductive L^2 estimate for w^p}
	\left\|\omega^{(p)}_{q+1}\right\|_{\bar L^2(\Omega,L^2_sL^2)}
	&\lesssim\left\|\mathring{R}^{(q)}\right\|_{\bar L^1(\Omega,L^1_sL^1)}^{1/2}+2^{-q}+C_q\left(\varsigma_{q+1}^{-\frac{1}{2}}\ell_{q+1}^{-1}+\sigma_{q+1}^{-\frac{1}{2}}\ell_{q+1}^{-\frac{3d+5}{2}}{+\ell_{q+1}^{\frac{1}{2}}}\right)\left(1+\sup_{s\geq0}\mathbb E\left\|\mathring{R}^{(q)}\right\|_{C_{s}L^1}^{5}\right)^{\frac{1}{2}}\nonumber\\
	&\lesssim\left\|\mathring{R}^{(q)}\right\|_{\bar L^1(\Omega,L^1_sL^1)}^{1/2}+2^{-q}+\mathfrak C_q\left(\lambda_{q+1}^{-\left(\frac{d}{2}+2\right)\vartheta}+\lambda_{q+1}^{-\left(\frac{1}{4\vartheta}-\frac{3d+5}{2}\vartheta\right)}{+\lambda_{q+1}^{-\vartheta/2}}\right)\nonumber\\
	&\lesssim\left\|\mathring{R}^{(q)}\right\|_{\bar L^1(\Omega,L^1_sL^1)}^{1/2}+2^{-q}+\mathfrak C_q\cdot\lambda_{q+1}^{-{\vartheta/2}}~.
	\end{align}
Here we use the first constraint of (\ref{setting of vartheta}) to deduce that
\begin{align}
	\frac{1}{4\vartheta}-\frac{3d+5}{2}\vartheta>\frac{4d+7}{2}\vartheta-\frac{3d+5}{2}\vartheta>\vartheta~.\nonumber
\end{align}
Then by (\ref{definition of corrector perturbation}), (\ref{estimates for Mikado flows V}) and (\ref{subtle estimates for a_k}), we  have for all $t\in[s,s+1]$ ($s\geq0$) that
\begin{align}\label{inductive spacial-L^infinity-estimate for w^c}
	\left\|\omega^{(c)}_{q+1}(t)\right\|_{L^\infty}
	&\lesssim\sigma_{q+1}^{-1}\cdot\sum_{k\in\Lambda}\left\|\nabla a_k^{(q+1)}(t)\right\|_{0}\left\|\mathbb V^{(q+1)}_k\right\|_{L^\infty}\nonumber\\
	&\leq C_q\cdot\left|g_{\kappa_{q+1}}\left(\varsigma_{q+1}t\right)\right|\cdot\sigma_{q+1}^{-1}\mu_{q+1}^{-1+\frac{d-1}{2}}\ell_{q+1}^{-\frac{3d+5}{2}}\left(1+\left\|\mathring{R}^{(q)}\right\|_{C_{[s-1,s+1]}L^1}^{5/2}\right)~.
\end{align}
Taking $L^2[s,s+1]$-norm in time, then the second moment and finally supremum for $s\geq0$, by (\ref{L^2 norm for g_kappa}), (\ref{finiteness of all moments of v^q and R^q}) and (\ref{setting of iteration parameters}), we have
\begin{align}\label{inductive L^2 estimate for w^c}
	\left\|\omega^{(c)}_{q+1}\right\|_{\bar L^2(\Omega,L^2_sL^2)}
	&\lesssim\left\|\omega^{(c)}_{q+1}\right\|_{\bar L^2(\Omega,L^2_sL^\infty)}\nonumber\\
	&\leq C_q\cdot\sigma_{q+1}^{-1}\mu_{q+1}^{-1+\frac{d-1}{2}}\ell_{q+1}^{-\frac{3d+5}{2}}\left(1+\sup_{s\geq0}\mathbb E\left\|\mathring{R}^{(q)}\right\|_{C_{s}L^1}^{3}\right)^{1/2}\nonumber\\
	&\leq C_q\cdot\lambda_{q+1}^{-\left(\frac{1}{2\vartheta}-\frac{3d+5}{2}\vartheta-\frac{d-3}{2}\right)}\nonumber\\
	&\leq C_q\cdot\lambda_{q+1}^{-\vartheta}~.
\end{align}
Here we use the first constraint of (\ref{setting of vartheta}) to deduce that
\begin{align}
	\frac{1}{2\vartheta}-\frac{3d+5}{2}\vartheta-\frac{d-3}{2}>\frac{5d+9}{2}\vartheta+1>\vartheta~.\nonumber
\end{align}
For $\omega^{(t)}_{q+1}$ we estimate {$W^{1,a}$ bound with $a>1$} for later use. By (\ref{definition of temporal perturbation}), (\ref{uniform bounds for h_kappa}) and using $L^a$-boundedness of the Helmholtz projection, mollification estimates and the Sobolev embedding $W^{2+\left(1-\frac{1}{a}\right)d,1}\hookrightarrow W^{2,a}$, we obtain for all $t\in[s,s+1]$ ($s\geq0$) that
\begin{align}\label{inductive spacial H^1,infinity estimate for w^t}
	\left\|\omega^{(t)}_{q+1}(t)\right\|_{W^{1,a}}
	&\lesssim\big|h_{\kappa}\left(\varsigma_{q+1}t\right)\big|\cdot\varsigma_{q+1}^{-1}\left\|\mathring{R}^{(q)}_{\ell}(t)\right\|_{W^{2,a}}\nonumber\\
	&\lesssim\varsigma_{q+1}^{-1}\left\|\mathring{R}^{(q)}_{\ell}(t)\right\|_{W^{2+\left(1-\frac{1}{a}\right)d,1}}\nonumber\\
	&\lesssim\varsigma_{q+1}^{-1}\ell_{q+1}^{-2-\left(1-\frac{1}{a}\right)d}\left\|\mathring{R}^{(q)}\right\|_{C_{[s-1,s+1]}L^1}\nonumber\\
	&\lesssim\varsigma_{q+1}^{-1}\ell_{q+1}^{-(d+2)}\left\|\mathring{R}^{(q)}\right\|_{C_{[s-1,s+1]}L^1}.
\end{align}
Then taking $L^2$-norm in time and in probability, and finally supremum for $s\geq0$, by (\ref{L^2 norm for g_kappa}), (\ref{finiteness of all moments of v^q and R^q}) and (\ref{setting of iteration parameters}), we have
\begin{align} \label{inductive L^2 estimate for w^t}
	\left\|\omega^{(t)}_{q+1}\right\|_{\bar L^2(\Omega,L^2_sL^2)}
	&\lesssim\left\|\omega^{(t)}_{q+1}\right\|_{\bar L^2(\Omega,L^2_sW^{1,2})}\nonumber\\
	&\lesssim\varsigma_{q+1}^{-1}\ell_{q+1}^{-(d+2)}\cdot\left(\sup_{s\geq0}\mathbb E\left\|\mathring{R}^{(q)}\right\|_{C_{s}L^1}^2\right)^{1/2}\nonumber\\
	&\leq C_q\cdot\lambda_{q+1}^{-\vartheta}~.
\end{align}
Hence, by (\ref{construction of v^(q+1)}), (\ref{def. cutoff Theta}), (\ref{inductive L^2 estimate for w^p}), (\ref{inductive L^2 estimate for w^c}), (\ref{inductive L^2 estimate for w^t}) and Proposition \ref{prop. uniform convergence of mollified z}, we have
\begin{align}\label{inductive L^2L^2 estimate for total perturbation}
	\big\|\varpi_{q+1}\big\|_{\bar L^2(\Omega,L^2_sL^2)}
	&\leq\left\|v^{(q)}_{\ell_{q+1}}-v^{(q)}\right\|_{\bar L^2(\Omega,L^2_sL^2)}+\left\|\omega^{(p)}_{q+1}\right\|_{\bar L^2(\Omega,L^2_sL^2)}+\left\|\omega^{(c)}_{q+1}\right\|_{\bar L^2(\Omega,L^2_sL^2)}+\left\|\omega^{(t)}_{q+1}\right\|_{\bar L^2(\Omega,L^2_sL^2)}\nonumber\\
	&\lesssim\left\|\mathring{R}^{(q)}\right\|_{\bar L^1(\Omega,L^1_sL^1)}^{1/2}+2^{-q}+ C_q\cdot\lambda_{q+1}^{-\vartheta/2}~.
\end{align}
Choose $\lambda_{q+1}$  sufficiently large and we get (\ref{L2 smallness of pert.}).

As the next step, we shall verify \eqref{weak smallness of pert.}. To this end, we estimate each norms in the definition of $Z^{p,r}_s$ and $\mathbf E_p$.

\noindent$\bullet$ {\bf $\bar L^\alpha\big(\Omega,C_sW^{-1,1}\big)$-Estimate :}\

	\par By (\ref{div-representation of w^p+w^c}), (\ref{estimates for Mikado flows V}), (\ref{subtle estimates for a_k}) and (\ref{definition of g_kappa}), we have for all $t\in[s,s+1]$ ($s\geq0$) that
	\begin{align}
		\left\|\omega_{q+1}^{(p)}(t)+\omega_{q+1}^{(c)}(t)\right\|_{W^{-1,1}}
		&\leq\sigma_{q+1}^{-1}\cdot\left\|\sum_{k\in\Lambda}a_k^{(q+1)}(t)\mathbb V_k^{(q+1)}\left(\sigma_{q+1}\cdot\right)\right\|_{L^m}\nonumber\\
		&\leq\sigma_{q+1}^{-1}\cdot\sum_{k\in\Lambda}\left\|a_k^{(q+1)}(t)\right\|_0\left\|\mathbb V_k^{(q+1)}\right\|_{L^m}\nonumber\\
		&\leq C_q\cdot\kappa_{q+1}^{1/2}\sigma_{q+1}^{-1}\mu_{q+1}^{-1+\frac{d-1}{2}-\frac{d-1}{m}}\ell_{q+1}^{-\frac{d+1}{2}}\left(1+\left\|\mathring{R}^{(q)}\right\|_{C_{[s-1,s+1]}L^1}^{3/2}\right).\nonumber
	\end{align}
Here we take $m>1$ and close to $1$.
	Taking supremum norm on $[s,s+1]$, then $\alpha$-th moment, and finally supremum for $s\geq0$, by (\ref{finiteness of all moments of v^q and R^q}) and (\ref{setting of iteration parameters}), we have
	\begin{align}\label{inductive CH^-2 estimate for w^p+w^c}
		\left\|\omega_{q+1}^{(p)}+\omega_{q+1}^{(c)}\right\|_{\bar L^\alpha(\Omega,C_sW^{-1,1})}
		&\leq C_q\cdot\kappa_{q+1}^{1/2}\sigma_{q+1}^{-1}\mu_{q+1}^{-1+\frac{d-1}{2}-\frac{d-1}{m}}\ell_{q+1}^{-\frac{d+1}{2}}\left(1+\sup_{s\geq0}\mathbb E\left\|\mathring{R}^{(q)}\right\|_{C_{s}L^1}^{\alpha/2}\right)^{1/\alpha}\nonumber\\
		&\leq C_q\cdot\lambda_{q+1}^{-(3d+8)\vartheta}\nonumber\\
		&\leq C_q\cdot\lambda_{q+1}^{-\vartheta}~.
	\end{align}
For $\omega_{q+1}^{(t)}$ by (\ref{definition of temporal perturbation}), (\ref{uniform bounds for h_kappa}) and the Sobolev embedding $W^{d/2,1}\hookrightarrow L^2$, we have for all $t\in[s,s+1]$ ($s\geq0$) that
	\begin{align}
		\left\|\omega_{q+1}^{(t)}(t)\right\|_{H^{-1}}&\leq\varsigma_{q+1}^{-1}\left|h_{\kappa_{q+1}}\left(\varsigma_{q+1}t\right)\right|\left\|\mathring{R}^{(q)}_{\ell}(t)\right\|_{L^2}\nonumber\\
		 &\lesssim\varsigma_{q+1}^{-1}\left\|\mathring{R}^{(q)}_{\ell}(t)\right\|_{W^{d/2,1}}\nonumber\\
		 &\lesssim\varsigma_{q+1}^{-1}\ell_{q+1}^{-d/2}\left\|\mathring{R}^{(q)}\right\|_{C_{[s-1,s+1]}L^1}~.\nonumber
	\end{align}
	Taking supremum norm on $[s,s+1]$ and then $\alpha$-th moment, and finally supremum for $s\geq0$, by (\ref{finiteness of all moments of v^q and R^q}) and (\ref{setting of iteration parameters}), we have
	\begin{align}\label{inductive CH^-2 estimate for w^t}
		\left\|\omega_{q+1}^{(t)}\right\|_{\bar L^\alpha(\Omega,C_sH^{-1})}
		&\lesssim\varsigma_{q+1}^{-1}\ell_{q+1}^{-d/2}\left(\sup_{s\geq0}\mathbb E\left\|\mathring{R}^{(q)}\right\|_{C_{s}L^1}^\alpha\right)^{1/\alpha}\nonumber\\
		&\leq C_q\cdot\lambda_{q+1}^{-\vartheta}~.
	\end{align}

Hence by {(\ref{def. cutoff Theta}),} (\ref{inductive CH^-2 estimate for w^p+w^c}), (\ref{inductive CH^-2 estimate for w^t}) and Proposition \ref{prop. uniform convergence of mollified z}, we have
\begin{align}\label{inductive CH^-1 estimate for total perturbation}
	\big\|\varpi_{q+1}\big\|_{\bar L^\alpha(\Omega,C_sW^{-1,1})}
	&\lesssim\left\|v^{(q)}_{\ell_{q+1}}-v^{(q)}\right\|_{\bar L^\alpha(\Omega,C_sL^\infty)}+\left\|\omega_{q+1}^{(p)}+\omega_{q+1}^{(c)}\right\|_{\bar L^\alpha(\Omega,C_sW^{-1,1})}+\left\|\omega_{q+1}^{(t)}\right\|_{\bar L^\alpha(\Omega,C_sH^{-1})}\nonumber\\
	&\leq C_q\cdot\lambda_{q+1}^{-\vartheta}~.
\end{align}\

\noindent$\bullet$ {\bf $\mathbf E_p$- and $\bar L^\alpha\big(\Omega,L^p_sL^\infty\big)$- Estimate :}\
\par The estimates of these two norms are similar and we only give the $\mathbf E_p$-estimate here.
	\par By (\ref{definition of principal pertutbation}), (\ref{estimates for Mikado flows W}) and (\ref{subtle estimates for a_k}), we have for all $t\in[s,s+1]$ ($s\geq0$) that
	\begin{align}\label{inductive spacial L^infinity estimate for w^p}
		\left\|\omega^{(p)}_{q+1}(t)\right\|_{L^\infty}
		&\leq\sum_{k\in\Lambda}\left\|a_k^{(q+1)}(t)\right\|_{0}\left\|\mathbb W^{(q+1)}_k\right\|_{L^\infty}\nonumber\\
	&\leq C_q\cdot\left|g_{\kappa_{q+1}}\left(\varsigma_{q+1}t\right)\right|\cdot\mu_{q+1}^{\frac{d-1}{2}}\ell_{q+1}^{-\frac{d+1}{2}}\left(1+\left\|\mathring{R}^{(q)}\right\|_{C_{[s-1,s+1]}L^1}^{3/2}\right)~.
	\end{align}
	Taking second moment, then $L^p[s,s+1]$-norm in time and finally supremum for $s\geq0$, by (\ref{L^alpha estimates for g_kappa}), (\ref{finiteness of all moments of v^q and R^q}), (\ref{setting of iteration parameters}), we have
	\begin{align}\label{inductive L^pL^2L^infinity estimate for w^p}
		\left\|\omega^{(p)}_{q+1}\right\|_{\mathbf E_p}
		&\leq C_q\cdot\kappa_{q+1}^{\frac{1}{2}-\frac{1}{p}}\mu_{q+1}^{\frac{d-1}{2}}\ell_{q+1}^{-\frac{d+1}{2}}\left(1+\sup_{s\geq0}\mathbb E\left\|\mathring{R}^{(q)}\right\|^{3}_{C_{s}L^1}\right)^{1/2}\nonumber\\
		&\leq C_q\cdot\lambda_{q+1}^{-\left[\left(\frac{1}{p}-\frac{1}{2}\right)\left(1+d-(5d+17)\vartheta+\frac{1}{\vartheta}\right)-\frac{d+1}{2}\vartheta -\frac{d-1}{2}\right]}\nonumber\\
		&\leq C_q\cdot\lambda_{q+1}^{-\vartheta}~.
	\end{align}
	Here, we use the fact that $1\leq p<2$, and the second and the third constraint of (\ref{setting of vartheta}) to deduce that
	\begin{align}
		\left(\frac{1}{p}-\frac{1}{2}\right)\left(1+d-(5d+17)\vartheta+\frac{1}{\vartheta}\right)-\frac{d+1}{2}\vartheta -\frac{d-1}{2}\geq\vartheta~.\nonumber
	\end{align}
	For $\omega^{(c)}_{q+1}$ and $\omega^{(t)}_{q+1}$, we use  (\ref{inductive spacial-L^infinity-estimate for w^c}) and (\ref{inductive spacial H^1,infinity estimate for w^t}) with $a=d+1$ and obtain
	\begin{align}\label{inductive L^pL^2L^infinity estimate for w^c}
		\left\|\omega^{(c)}_{q+1}\right\|_{\mathbf E_p}
		&\leq C_q\cdot\kappa_{q+1}^{\frac{1}{2}-\frac{1}{p}}\sigma_{q+1}^{-1}\mu_{q+1}^{-1+\frac{d-1}{2}}\ell_{q+1}^{-\frac{3d+5}{2}}\left(1+\sup_{s\geq0}\mathbb E\left\|\mathring{R}^{(q)}\right\|_{C_{s}L^1}^{5}\right)^{1/2}\nonumber\\
		&\leq C_q\cdot\sigma_{q+1}^{-1}\mu_{q+1}^{-1+\frac{d-1}{2}}\ell_{q+1}^{-\frac{3d+5}{2}}\nonumber\\
		&\leq C_q\cdot\lambda_{q+1}^{-\vartheta}~,\\
	\label{inductive L^pL^2L^infinity estimate for w^t}
		\left\|\omega^{(t)}_{q+1}\right\|_{\mathbf E_p}
		&\lesssim\sup_{s\geq0}\left\|\omega^{(t)}_{q+1}\right\|_{L^p_sL^2\left(\Omega; W^{1,d+1}\right)}\nonumber\\
		&\lesssim\varsigma_{q+1}^{-1}\ell_{q+1}^{-(d+2)}\left(\sup_{s\geq0}\mathbb E\left\|\mathring{R}^{(q)}\right\|_{C_{s}L^1}^2\right)^{1/2}\nonumber\\
		&\leq C_q\cdot\lambda_{q+1}^{-\vartheta}~.
	\end{align}
	
Hence, by {(\ref{def. cutoff Theta}),} (\ref{inductive L^pL^2L^infinity estimate for w^p}), (\ref{inductive L^pL^2L^infinity estimate for w^c}), (\ref{inductive L^pL^2L^infinity estimate for w^t}) and Proposition \ref{prop. uniform convergence of mollified z}, we have
	\begin{align}\label{inductive L^pL^2L^infinity estimate for total perturbation}
	\big\|\varpi_{q+1}\big\|_{\mathbf E_p}
	&\leq\left\|v^{(q)}_{\ell_{q+1}}-v^{(q)}\right\|_{\mathbf E_p}+\left\|\omega^{(p)}_{q+1}\right\|_{\mathbf E_p}+\left\|\omega^{(c)}_{q+1}\right\|_{\mathbf E_p}+\left\|\omega^{(t)}_{q+1}\right\|_{\mathbf E_p}\nonumber\\
	&\leq C_q\cdot\lambda_{q+1}^{-\vartheta}~.
\end{align}\

\noindent$\bullet$ {\bf $\bar L^\alpha\big(\Omega,L^1_sW^{1,r}\big)$-Estimate :}

By (\ref{definition of principal pertutbation}), (\ref{estimates for Mikado flows W}) and (\ref{subtle estimates for a_k}), we have for all $t\in[s,s+1]$ ($s\geq0$) that
\begin{align}
	\left\|\omega^{(p)}_{q+1}(t)\right\|_{W^{1,r}}
	&\lesssim\sum_{k\in\Lambda}\left\|a_k^{(q+1)}(t)\right\|_{1}\left\|\mathbb W^{(q+1)}_k\left(\sigma_{q+1}\cdot\right)\right\|_{W^{1,r}}\nonumber\\
	&\leq C_q\cdot\left|g_{\kappa_{q+1}}\left(\varsigma_{q+1}t\right)\right|\cdot\sigma_{q+1}\mu_{q+1}^{1+\frac{d-1}{2}-\frac{d-1}{r}}\ell_{q+1}^{-\frac{3d+5}{2}}\left(1+\left\|\mathring{R}^{(q)}\right\|_{C_{[s-1,s+1]}L^1}^{5/2}\right).\nonumber
\end{align}
Taking $L^1[s,s+1]$-norm in time, then $\alpha$-th moment, and finally supremum for $s\geq0$, by (\ref{L^alpha estimates for g_kappa}), (\ref{setting of iteration parameters}) and (\ref{finiteness of all moments of v^q and R^q}), we have
\begin{align}\label{inductive L^1H^1,a estimate for w^p}
	\left\|\omega^{(p)}_{q+1}\right\|_{\bar L^\alpha(\Omega,L^1_sW^{1,r})}
	&\leq C_q\cdot\kappa_{q+1}^{-\frac{1}{2}}\sigma_{q+1}\mu_{q+1}^{1+\frac{d-1}{2}-\frac{d-1}{r}}\ell_{q+1}^{-\frac{3d+5}{2}}\left(1+\sup_{s\geq0}\mathbb E\left\|\mathring{R}^{(q)}\right\|_{C_{s}L^1}^{5\alpha/2}\right)^{1/\alpha}\nonumber\\
	&\leq C_q\cdot\lambda_{q+1}^{-\left(-(4d+11)\vartheta+\frac{d-1}{r}\right)}\nonumber\\
	&\leq C_q\cdot\lambda_{q+1}^{-\vartheta}~.
\end{align}
Here we use the last constraint of (\ref{setting of vartheta}) to deduce that
\begin{align}
	-(4d+11)\vartheta+\frac{d-1}{r}\geq\vartheta~.\nonumber
\end{align}
Then by (\ref{definition of corrector perturbation}), (\ref{estimates for Mikado flows V}) and (\ref{subtle estimates for a_k}), we have for all $t\in[s,s+1]$ ($s\geq0$) that
\begin{align}
	\left\|\omega^{(c)}_{q+1}(t)\right\|_{W^{1,r}}
	&\lesssim\sigma_{q+1}^{-1}\sum_{k\in\Lambda}\left\|\nabla a_k^{(q+1)}(t)\right\|_{1}\left\|\mathbb V^{(q+1)}_k\left(\sigma_{q+1}\cdot\right)\right\|_{W^{1,r}}\nonumber\\
	&\leq C_q\cdot\left|g_{\kappa_{q+1}}\left(\varsigma_{q+1}t\right)\right|\cdot\mu_{q+1}^{\frac{d-1}{2}-\frac{d-1}{r}}\ell_{q+1}^{-\frac{5d+9}{2}}\left(1+\left\|\mathring{R}^{(q)}\right\|_{C_{[s-1,s+1]}L^1}^{7/2}\right).\nonumber
\end{align}
Taking $L^1[s,s+1]$-norm in time, then $\alpha$-th moment, and finally supremum for $s\geq0$, by (\ref{L^alpha estimates for g_kappa}), (\ref{setting of vartheta}) and (\ref{finiteness of all moments of v^q and R^q}), we have
\begin{align}\label{inductive L^1H^1,a estimate for w^c}
	\left\|\omega^{(c)}_{q+1}\right\|_{\bar L^\alpha(\Omega,L^1_sW^{1,r})}
	&\leq C_q\cdot\kappa_{q+1}^{-\frac{1}{2}}\mu_{q+1}^{\frac{d-1}{2}-\frac{d-1}{r}}\ell_{q+1}^{-\frac{5d+9}{2}}\left(1+\sup_{s\geq0}\mathbb E\left\|\mathring{R}^{(q)}\right\|_{C_{s}L^1}^{7\alpha/2}\right)^{1/\alpha}\nonumber\\
	&\leq C_q\cdot\lambda_{q+1}^{-\left(1+\frac{1}{2\vartheta}-(5d+13)\vartheta+\frac{d-1}{r}\right)}\nonumber\\
	&\leq C_q\cdot\lambda_{q+1}^{-\vartheta}~.
\end{align}
Here we use the last and the first constraint of (\ref{setting of vartheta}) to deduce
\begin{align}
	1+\frac{1}{2\vartheta}-(5d+13)\vartheta+\frac{d-1}{r}\geq1+\frac{1}{2\vartheta}-(d+1)\vartheta>\vartheta~.\nonumber
\end{align}
For $\omega^{(t)}_{q+1}$, we simply use (\ref{inductive spacial H^1,infinity estimate for w^t}) to have
\begin{align}\label{inductive L^1H^1,a estimate for w^t}
	\left\|\omega^{(t)}_{q+1}\right\|_{\bar L^\alpha(\Omega,L^1_sW^{1,r})}
		&\lesssim\varsigma_{q+1}^{-1}\ell_{q+1}^{-(d+2)}\left(\sup_{s\geq0}\mathbb E\left\|\mathring{R}^{(q)}\right\|_{C_{s}L^1}^{\alpha}\right)^{1/{\alpha}}\nonumber\\
		&\leq C_q\cdot\lambda_{q+1}^{-\vartheta}~.
\end{align}
Hence, by {(\ref{def. cutoff Theta}),} (\ref{inductive L^1H^1,a estimate for w^p}), (\ref{inductive L^1H^1,a estimate for w^c}), (\ref{inductive L^1H^1,a estimate for w^t}) and Proposition \ref{prop. uniform convergence of mollified z}, we have
\begin{align}\label{inductive L^1H^1 estimate for total perturbation}
	&~\big\|\varpi_{q+1}\big\|_{\bar L^\alpha(\Omega,L^1_sW^{1,r})}\nonumber\\
	&\leq\left\|v^{(q)}_{\ell_{q+1}}-v^{(q)}\right\|_{\bar L^\alpha(\Omega,L^1_sW^{1,r})}+\left\|\omega^{(p)}_{q+1}\right\|_{\bar L^\alpha(\Omega,L^1_sW^{1,r})}+\left\|\omega^{(c)}_{q+1}\right\|_{\bar L^\alpha(\Omega,L^1_sW^{1,r})}+\left\|\omega^{(t)}_{q+1}\right\|_{\bar L^\alpha(\Omega,L^1_sW^{1,r})}\nonumber\\
	&\leq C_q\cdot\lambda_{q+1}^{-\vartheta}~.
\end{align}

\par Combining (\ref{inductive L^2L^2 estimate for total perturbation}), (\ref{inductive CH^-1 estimate for total perturbation}), (\ref{inductive L^pL^2L^infinity estimate for total perturbation}) and (\ref{inductive L^1H^1 estimate for total perturbation}), and choosing $\lambda_{q+1}$ sufficiently large, we get  (\ref{weak smallness of pert.}).

\subsubsection{Inductive Estimate of $\mathring{R}^{(q+1)}$}\

To conclude the proof of Proposition, we shall verify \eqref{smallness of Rey.}. To this end, we estimate each term in (\ref{construction for R^(q+1)}) separately.\\
$\mathbf{(1)}~$ $\mathring{R}_{com}^\ell$~$\bf and$ $\mathring{R}_{com}^{(q+1)}$
	
\noindent$\bullet~{\bf Estimate~of~}\mathring{R}_{com}^\ell:$

By (\ref{definition of molli-communication Reynold stress}) and  mollification estimate, we  have
	\begin{align}
		\left\|\mathring{R}_{com}^\ell\right\|_{L^1_{s}L^1}
		&\lesssim\left\|v^{(q)}+z\right\|_{L^2_{[s-1,s+1]}L^2}\left(\left\|v^{(q)}_{\ell}-v^{(q)}\right\|_{L^2_{s}L^2}+\left\|z_{\ell}-z\right\|_{L^2_{s}L^2}\right)\nonumber\\
		&~~~+\left\|\left(\big(v^{(q)}+z\big){\mathring\otimes}\big(v^{(q)}+z\big)\right)_{\ell}-\big(v^{(q)}+z\big){\mathring\otimes}\big(v^{(q)}+z\big)\right\|_{L^1_{s}L^1},~~~~s\geq0\nonumber
	\end{align}
{Here, $\Big(\left(v^{(q)}+z\right)\mathring\otimes\left(v^{(q)}+z\right)\Big)_\ell:=\left[\Big(\left(v^{(q)}+z\right)\mathring\otimes\left(v^{(q)}+z\right)\Big)\ast_x\phi_\ell\right]\ast_t\varphi_\ell$ .}
	 Taking expectation, then by H\"older's inequality, (\ref{finiteness of all moments of v^q and R^q}) and Proposition \ref{prop. uniform convergence of mollified z}, we have
	\begin{align}
		\left\|\mathring{R}_{com}^\ell\right\|_{\bar L^1(\Omega,L^1_sL^1)}\longrightarrow0\nonumber
	\end{align}
	as $\ell=\ell_{q+1}\rightarrow0$. Hence, choose $\lambda_{q+1}$  sufficiently large so that $\ell_{q+1}=\lambda_{q+1}^{-\vartheta}$ sufficiently small and we obtain
	\begin{align}
		\left\|\mathring{R}_{com}^\ell\right\|_{\bar L^1(\Omega,L^1_sL^1)}\leq\frac{\delta^2}{4}~.\nonumber
	\end{align}\\
\noindent$\bullet~{\bf Estimate~of~}\mathring{R}_{com}^{(q+1)}:$

By (\ref{definition of R_com}) and  mollification estimate, we have for all $s\geq0$ that
	 \begin{align}
	 	\left\|\mathring{R}_{com}^{(q+1)}\right\|_{L^1_{s}L^1}
	 	&\lesssim\left\|v^{(q+1)}+z_\ell\right\|_{L^2_{s}L^2}\big\|z-z_\ell\big\|_{L^2_{s}L^2}+\big\|z-z_\ell\big\|_{L^2_{s}L^2}^2\nonumber\\
	 	&\lesssim\bigg(\left\|v^{(q)}\right\|_{0,[s-1,s+1]}+\big\|\omega_{q+1}\big\|_{L^2_{s}L^2}+\big\|z\big\|_{L^2_{[s-1,s+1]}L^2}\bigg)\big\|z-z_\ell\big\|_{L^2_{s}L^2}~.\nonumber
	 \end{align}
	 Taking expectation and  supremum for $s\geq0$, then by H\"older's inequality, (\ref{z finiteness of moments}), (\ref{finiteness of all moments of v^q and R^q}) and the previous $\bar L^2\big(\Omega,L^2_sL^2\big)$-estimates for perturbations, we have
	 \begin{align}
	 	\left\|\mathring{R}_{com}^{(q+1)}\right\|_{\bar L^1(\Omega,L^1_sL^1)}
	 	&\lesssim\left(\sup_{s\geq0}\mathbb E\left\|v^{(q)}\right\|_{0,s}^2+ C_q+\sup_{s\geq0}\mathbb E\big\|z\big\|_{L^2_{s}L^2}^2\right)^{1/2}\big\|z-z_{\ell}\big\|_{\bar L^2(\Omega,L^2_sL^2)}\nonumber\\
	 	&\leq C_q\cdot\big\|z-z_{\ell}\big\|_{\bar L^2(\Omega,L^2_sL^2)}\leq \frac{\delta^2}{4}~.\nonumber
	 \end{align}\\
	Here, in the last inequality, {by (\ref{uniform convergence of mollified z})},  we choose $\lambda_{q+1}$ large enough and obtain
	\begin{align}
		 C_q\cdot\big\|z-z_{\ell}\big\|_{\bar L^2(\Omega,L^2_sL^2)}\leq\frac{\delta^2}{4}~.\nonumber
	\end{align}
	\\
	$\mathbf{(2)~Oscillation~Errors}~\mathring R_{far}^{(q+1)}$, $\mathring{R}_{osc,x}^{(q+1)}$ and $\mathring{R}_{osc,t}^{(q+1)}$
	
\noindent$\bullet~{\bf Estimate~of~}\mathring R_{far}^{(q+1)}:$

By (\ref{definition of R_far}), (\ref{def. cutoff Theta}), (\ref{estimates for Mikado flows W0W}) and (\ref{subtle estimates for a_k}), we have for all $t\in{[s,s+1]}$ ($s\geq0$) that
	\begin{align}
		\left\|\mathring{R}_{far}^{(q+1)}(t)\right\|_{L^1}
		&\lesssim\sum_{k\neq k'}\left\|a_k^{(q+1)}(t)\right\|_{0}\left\|a_{k'}^{(q+1)}(t)\right\|_{0}\left\|\mathbb W_k^{(q+1)}\mathring\otimes \mathbb W_{k'}^{(q+1)}\right\|_{L^1}\nonumber\\
		&\leq C_q\cdot\left|g_{\kappa_{q+1}}\left(\varsigma_{q+1}t\right)\right|^2\cdot\mu_{q+1}^{-1}\ell_{q+1}^{-(d+1)}\left(1+\left\|\mathring{R}^{(q)}\right\|^{3}_{C_{[s-1,s+1]}L^1}\right)~.\nonumber
	\end{align}
	Taking $L^1{[s,s+1]}$-norm in time, then expectation and finally supremum for $s\geq0$, by (\ref{L^2 norm for g_kappa}), (\ref{finiteness of all moments of v^q and R^q}) and (\ref{setting of iteration parameters}), we have
	\begin{align}
		\left\|\mathring{R}_{far}^{(q+1)}\right\|_{\bar L^1(\Omega,L^1_sL^1)}
		&\leq C_q\cdot\mu_{q+1}^{-1}\ell_{q+1}^{-(d+1)}\left(1+\sup_{s\geq0}\mathbb E\left\|\mathring{R}^{(q)}\right\|^{3}_{C_{s}L^1}\right)\nonumber\\
		&\leq C_q\cdot\lambda_{q+1}^{-\left(1-(d+1)\vartheta\right)}\nonumber\\
		&\leq C_q\cdot\lambda_{q+1}^{-\vartheta}~.\nonumber
	\end{align}
	Here we have used the fact that $0<\vartheta<\frac{1}{d+2}$ .\\

\noindent$\bullet~{\bf Estimate~of~}\mathring{R}_{osc,x}^{(q+1)}:$

By (\ref{definition of R_osc,x}), (\ref{def. cutoff Theta}), (\ref{L^p estimate for bilinear reverse divergence}), $L^1$-boundedness of the operator $\mathcal R$, (\ref{estimates for Mikado flows W}), {(\ref{peri-L^p estimate of anti-divergence R}) and} (\ref{subtle estimates for a_k}), we have for all $t\in{[s,s+1]}$ ($s\geq0$) that
	 \begin{align}
	 	\left\|\mathring{R}_{osc,x}^{(q+1)}(t)\right\|_{L^1}
	 	&\lesssim\sum_{k\in\Lambda}\left\|\nabla\left|a_k^{(q+1)}(t)\right|^2\right\|_{1}\left\|\mathcal R\left[\mathbb W_k^{(q+1)}\left(\sigma_{q+1}\cdot\right)\mathring\otimes \mathbb W_k^{(q+1)}\left(\sigma_{q+1}\cdot\right)\right]\right\|_{L^1}\nonumber\\
	 	&\lesssim\sigma_{q+1}^{-1}\sum_{k\in\Lambda}\left\|\nabla a_k^{(q+1)}(t)\right\|_{1}\left\|a_k^{(q+1)}(t)\right\|_{1}\left\|\mathbb W_k^{(q+1)}\right\|_{L^{2}}^2\nonumber\\
	 	&\leq C_q\cdot\left|g_{\kappa_{q+1}}\left(\varsigma_{q+1}t\right)\right|^2\cdot\sigma_{q+1}^{-1}\ell_{q+1}^{-(4d+7)}\left(1+\left\|\mathring{R}^{(q)}\right\|_{C_{[s-1,s+1]}L^1}^{6}\right)~.\nonumber
	 \end{align}

	 Taking $L^1{[s,s+1]}$-norm in time, and expectation and finally supremum for $s\geq0$, by the normalized $L^2$-norm of $g_\kappa$, (\ref{finiteness of all moments of v^q and R^q}) and (\ref{setting of iteration parameters}), we have
	  \begin{align}
	  	\left\|\mathring{R}_{osc,x}^{(q+1)}\right\|_{\bar L^1(\Omega,L^1_sL^1)}
	  	&\leq C_q\cdot\sigma_{q+1}^{-1}\ell_{q+1}^{-(4d+7)}\left(1+\sup_{s\geq0}\mathbb E\left\|\mathring{R}^{(q)}\right\|_{C_{s}L^1}^{6}\right)\nonumber\\
	  	&\leq C_q\cdot\lambda_{q+1}^{-\left(\frac{1}{2\vartheta}-(4d+7)\vartheta\right)}\nonumber\\
		&\leq C_q\cdot\lambda_{q+1}^{-\vartheta}~.\nonumber
	  \end{align}
	  Here we use the first constraint of (\ref{setting of vartheta}) to deduce that
	  \begin{align}
	  	\frac{1}{2\vartheta}-(4d+7)\vartheta\geq\frac{d-1}{2}>\vartheta~.\nonumber	
	 \end{align}\\
\noindent$\bullet~{\bf Estimate~of~}\mathring{R}_{osc,t}^{(q+1)}:$

By (\ref{definition for R_osc,t}), {(\ref{def. cutoff Theta}),} (\ref{uniform bounds for h_kappa}) and  mollification estimate, we have for all $t\in{[s,s+1]}$ ($s\geq0$) that
	 \begin{align}
	 	\left\|\mathring{R}_{osc,t}^{(q+1)}(t)\right\|_{L^1}&\leq\varsigma_{q+1}^{-1}\left|h_{\kappa_{q+1}}\left(\varsigma_{q+1}t\right)\right|\left\|\partial_t\mathring{R}^{(q)}_{\ell_{q+1}}(t)\right\|_{L^1}\nonumber\\
	 	&\lesssim\varsigma_{q+1}^{-1}\ell_{q+1}^{-1}\left\|\mathring{R}^{(q)}\right\|_{C_{[s-1,s+1]}L^1}~.\nonumber
	 \end{align}
	Taking $L^1{[s,s+1]}$-norm in time, then expectation and finally supremum for $s\geq0$, by (\ref{finiteness of all moments of v^q and R^q}) and (\ref{setting of iteration parameters}), we have
	 \begin{align}
	 	\left\|\mathring{R}_{osc,t}^{(q+1)}\right\|_{\bar L^1(\Omega,L^1_sL^1)}
	 	&\lesssim\varsigma_{q+1}^{-1}\ell_{q+1}^{-1}\cdot\sup_{s\geq0}\mathbb E\left\|\mathring{R}^{(q)}\right\|_{C_{s}L^1}~,\nonumber\\
		&\leq C_q\cdot\lambda_{q+1}^{-\vartheta}~.\nonumber
	 \end{align}
	\\
$\mathbf{(3)~Linear~Error}~\mathring{R}_{lin}^{(q+1)}$:
\par In the following we choose $\gamma:=\frac{2(d-1)}{2(d-1)-\vartheta}>1$, i.e. $(d-1)-\frac{d-1}{\gamma}=\frac{\vartheta}{2}$. We estimate each term in (\ref{definition of R_lin}) separately. First, using (\ref{div-representation of w^p+w^c}), {(\ref{def. cutoff Theta}), and $L^\gamma$-boundedness of the operator $\mathcal{R}{\rm div}$, (\ref{estimates for Mikado flows V}), (\ref{subtle estimate for partial_t a_k}), we have for all $t\in{[s,s+1]}$ ($s\geq0$) that
	\begin{align}
	 	&~~~~\left\|\mathcal R\left[\Theta_{q+1}(t)\partial_t\left(\omega^{(p)}_{q+1}(t)+\omega^{(c)}_{q+1}(t)\right)\right]\right\|_{L^\gamma}\nonumber\\
	 	&\lesssim\sigma_{q+1}^{-1}\sum_{k\in\Lambda}\left\|\partial_t a_k^{(q+1)}(t)\mathbb V_k^{(q+1)}\left(\sigma_{q+1}\cdot\right)\right\|_{L^\gamma}\nonumber\\
	 	&\lesssim\sigma_{q+1}^{-1}\sum_{k\in\Lambda}\left\|\partial_t a_k^{(q+1)}(t)\right\|_{0}\left\|\mathbb V_k^{(q+1)}\left(\sigma_{q+1}\cdot\right)\right\|_{L^\gamma}\nonumber\\
	 	&\leq C_q\cdot\sigma_{q+1}^{-1}\mu_{q+1}^{-1+\frac{d-1}{2}-\frac{d-1}{\gamma}}\Big(\varsigma_{q+1}\big|g'_{\kappa}\left(\varsigma_{q+1}t\right)\big|+\big|g_{\kappa}\left(\varsigma_{q+1}t\right)\big|\ell_{q+1}^{-1}\Big)\ell_{q+1}^{-\frac{3d+3}{2}}\left(1+\left\|\mathring{R}^{(q)}\right\|_{C_{[s-1,s+1]}L^1}^{5/2}\right)~.\nonumber
	\end{align}
Taking $L^1{[s,s+1]}$-norm in time, then expectation and finally supremum for $s\geq0$, by (\ref{L^alpha estimates for g_kappa}), (\ref{finiteness of all moments of v^q and R^q}) and (\ref{setting of iteration parameters}), we have
	\begin{align}\label{inductive L^1 estimate of R_lin, 1}
	 	&~\left\|\mathcal R\left[\Theta_{q+1}\partial_t\left(\omega^{(p)}_{q+1}+\omega^{(c)}_{q+1}\right)\right]\right\|_{\bar L^1(\Omega,L^1_sL^1)}\nonumber\\
	 	&\leq C_q\cdot\sigma_{q+1}^{-1}\mu_{q+1}^{-1+\frac{d-1}{2}-\frac{d-1}{\gamma}}\ell_{q+1}^{-\frac{3d+3}{2}}\left(\kappa_{q+1}^{\frac{1}{2}}\varsigma_{q+1}+\kappa_{q+1}^{-\frac{1}{2}}\ell_{q+1}^{-1}\right)\left(1+\sup_{s\geq0}\mathbb E\left\|\mathring{R}^{(q)}\right\|_{C_{s}L^1}^{5/2}\right)\nonumber\\
	 	&\leq C_q\cdot\lambda_{q+1}^{-\left(\vartheta-(d-1)+\frac{d-1}{\gamma}\right)}\nonumber\\
		&\leq C_q\cdot\lambda_{q+1}^{-\vartheta/2}~.
	\end{align}
Here, we use the definition of $\gamma$ and  to deduce that
	 \begin{align}
	    &\vartheta-(d-1)+\frac{d-1}{\gamma}=\vartheta-\frac{\vartheta}{2}=\frac{\vartheta}{2}~.\nonumber
	 \end{align}
For the second part, we use {(\ref{relation R-Laplace}) and} the previous $\bar L^\alpha\big(\Omega,L^1_sW^{1,r}\big)$-estimates for perturbations to obtain
	\begin{align}\label{inductive L^1 estimate of R_lin, 2}
	 	\left\|\mathcal{R}\Big(\nu\Delta\omega_{q+1}+(z_{\ell}-z)\Big)\right\|_{\bar L^1(\Omega,L^1_sL^1)}
	 	&\lesssim\big\|\omega_{q+1}\big\|_{\bar L^1(\Omega,L^1_sH^1)}+\big\|z_{\ell}-z\big\|_{\bar L^2(\Omega,L^2_sL^2)}\nonumber\\
	 	&\leq C_q\cdot\lambda_{q+1}^{-\vartheta}+\frac{\delta^2}{4} ~.
	\end{align}
Here we used similar argument as above and choose $\lambda_{q+1}$  sufficiently large so that }
	\begin{align}
		\big\|z_{\ell}-z\big\|_{\bar L^2(\Omega,L^2_sL^2)}\leq\frac{\delta^2}{4}~.\nonumber
	\end{align}
For the last part, by using (\ref{z finiteness of moments}), (\ref{finiteness of all moments of v^q and R^q}) and the previous $\bar L^\alpha\big(\Omega,L^p_sL^\infty\big)$-estimates for perturbations we obtain
	\begin{align}\label{inductive L^1 estimate of R_lin, 3}
	 	&~\left\|\left(v^{(q)}_\ell+z_\ell\right)\mathring\otimes~\omega_{q+1}+\omega_{q+1}\mathring\otimes\left(v^{(q)}_\ell+z_\ell\right)\right\|_{\bar L^1(\Omega,L^1_sL^1)}\nonumber\\
	 	&\lesssim\left(\sup_{s\geq0}\mathbb E\left\|v^{(q)}\right\|_{0,[s,s+2]}^2+\sup_{s\geq0}\mathbb E\big\|z\big\|_{C_{[s,s+2]}L^2}^2\right)^{1/2}\big\|\omega_{q+1}\big\|_{\bar L^2(\Omega,L^p_sL^\infty)}~,\nonumber\\
	 	&\leq C_q\cdot\lambda_{q+1}^{-\vartheta}~.
	\end{align}
	 Hence, by (\ref{inductive L^1 estimate of R_lin, 1}), (\ref{inductive L^1 estimate of R_lin, 2}) and (\ref{inductive L^1 estimate of R_lin, 3}), we obtain
	 \begin{align}
	 	\left\|\mathring{R}_{lin}^{(q+1)}\right\|_{\bar L^1(\Omega,L^1_sL^1)}\leq C_q\cdot\lambda_{q+1}^{-\vartheta/2}+\frac{\delta^2}{4} ~.\nonumber
	 \end{align}
	 \\
$\mathbf{(4)~Correction~Error}~\mathring{R}_{cor}^{(q+1)}$:\\
\par By (\ref{definition of R_cor}), (\ref{def. cutoff Theta}), (\ref{inductive L^2 estimate for w^p}), (\ref{inductive L^2 estimate for w^c}), (\ref{inductive L^2 estimate for w^t}) and (\ref{finiteness of all moments of v^q and R^q}), we have
	\begin{align}
	 	\left\|\mathring{R}_{cor}^{(q+1)}\right\|_{\bar L^1(\Omega,L^1_sL^1)}
	 	&\lesssim\left(\big\|\omega_{q+1}\big\|_{\bar L^2(\Omega,L^2_sL^2)}+\big\|\omega^{(p)}_{q+1}\big\|_{\bar L^2(\Omega,L^2_sL^2)}\right)\left\|\omega^{(c)}_{q+1}+\omega^{(t)}_{q+1}\right\|_{\bar L^2(\Omega,L^2_sL^2)}\nonumber\\
	 	&\lesssim\left(\left\|\mathring{R}^{(q)}\right\|_{\bar L^1(\Omega,L^1_sL^1)}^{1/2}+2^{-q}+ C_q\cdot\lambda_{q+1}^{-\vartheta}\right)\cdot C_q\cdot\lambda_{q+1}^{-\vartheta}\nonumber\\
	 	&\leq C_q\cdot\lambda_{q+1}^{-\vartheta}~.\nonumber
	\end{align}
$\mathbf{(5)~Cutoff~Error}~\mathring{R}_{cut}^{(q+1)}$:
\par By (\ref{definition of R_cut}), (\ref{def. cutoff Theta}), standard mollification estimate and $L^1$-boundedness of the operator $\mathcal{R}$, we have for all $t\in[s,s+1]$ ($s\geq0$) that
	 \begin{align}
	 	&~\left\|\mathring{R}_{cut}^{(q+1)}(t)\right\|_{L^1}\nonumber\\
	 	&\lesssim\Big(1-\Theta_{q+1}^2(t)\Big)\left\|\mathring{R}^{(q)}\right\|_{C_{[s-1,s+1]}L^1}+{\big\|\Theta_{q+1}'\big\|_0}\bigg(\left\|\omega^{(p)}_{q+1}(t)\right\|_{{L^1}}+\left\|\omega^{(c)}_{q+1}(t)\right\|_{L^1}+\left\|\omega^{(t)}_{q+1}(t)\right\|_{L^1}\bigg)~.\nonumber
	 \end{align}
	 Taking $L^1{[s,s+1]}$-norm in time and expectation and supremum for $s\geq0$, by (\ref{finiteness of all moments of v^q and R^q}), H\"older's inequality, the previous {$\bar L^1\big(\Omega,L^1_sW^{1,r}\big)$}-estimates of perturbations {and (\ref{def. cutoff Theta})},
we have
	 \begin{align}
	 	&~\left\|\mathring{R}_{cut}^{(q+1)}\right\|_{\bar L^1(\Omega,L^1_sL^1)}\nonumber\\
	 	&\lesssim{\ell_{q+1}^{1/2}}\cdot\sup_{s\geq0}\mathbb E\left\|\mathring{R}^{(q)}\right\|_{C_{s}L^1} +{\big\|\Theta_{q+1}'\big\|_{0}\left(\left\|\omega^{(p)}_{q+1}\right\|_{\bar L^1(\Omega,L^1_sW^{1,r})}+\left\|\omega^{(c)}_{q+1}\right\|_{\bar L^1(\Omega,L^1_sW^{1,r})}+\left\|\omega^{(t)}_{q+1}\right\|_{\bar L^1(\Omega,L^1_sW^{1,r})}\right)}\nonumber\\
	 	&\leq {C_q\cdot\lambda_{q+1}^{-\vartheta/2}+\ell_{q+1}^{-1/2}\cdot C_q\cdot\lambda_{q+1}^{-\vartheta}\leq C_q\cdot\lambda_{q+1}^{-\vartheta/2}}~.\nonumber
	 \end{align}\
\par Hence, we finally have
\begin{align}
	 \left\|\mathring{R}^{(q+1)}\right\|_{\bar L^1(\Omega,L^1_sL^1)}\leq C_q\cdot\lambda_{q+1}^{-{\vartheta/2}}+\frac{3\delta^2}{4}~.\nonumber
\end{align}
Choosing $\lambda_{q+1}$  sufficiently large, we get (\ref{smallness of Rey.}). The proof of Proposition \ref{prop. main iteration} is complete.\\

\section{Proof of Theorem \ref{thm. Stochastic Serrin's Criterion}}\label{sec. sto. weak-strong uniqueness}\	

\par In this section we give the proof of Theorem \ref{thm. Stochastic Serrin's Criterion} by extending the classical {result \cite{FJR72, Kat84, FLRT00, LM01, CL22}} in PDE case to the stochastic case. For simplicity we consider \eqref{sto. NS} on $[0,T]$ for any $T>0$.
In the following we first prove that a $X^{p,q}_T$-valued solution is a Leray-Hopf solution. Then we prove pathwise uniqueness. To this end, we introduced the following stochastic linearized N-S system:
\begin{equation} \label{linearized sto. N-S}
	\left\{\begin{aligned}
		&{\rm d}\chi(t)=\Big(-{\rm div}\big(u(t)\otimes\chi(t)\big) + \Delta\chi(t)-\nabla p(t)\Big){\rm d}t+{\rm dW}(t)~,\\
		&{\rm div}\ \chi=0~.\\
	\end{aligned}\right.
\end{equation}

\begin{theorem}\label{thm. existence of linearized sto. N-S}
	Let $u\in X_T^{p,q}$ $P$-a.s. with $p,q$ satisfying \eqref{Serrin scales} be a solution to system (\ref{sto. NS}). Then for any given $\chi_0\in L^2_\sigma$ , there exists a probabilistically strong and analytically weak solution $\chi\in L^2\Big(\Omega;~C_{[0,T]}L^2\Big)\\\bigcap L^2\Big(\Omega;~L^2_{[0,T]}H^1\Big)$ to system (\ref{linearized sto. N-S}) with initial data $\chi_0$ such that $\mathbb P-a.s.$ ,
	\begin{align}\label{expectational energy inequality for linearized sto. N-S}
		\frac{1}{2}\mathbb E\big\|\chi(t)\big\|_{L^2}^2+\nu\int_0^t\mathbb E\big\|\nabla\chi(s)\big\|_{L^2}^2ds\leq\frac{1}{2}\mathbb E\big\|\chi_0\big\|_{L^2}^2+\frac{t}{2}{\rm Tr}\big[G^*G\big]~,~t\in[0,T]~.
	\end{align}
\end{theorem}\

\par Existence follows by a standard Galerkin method and we omit the details. The uniqueness of probabilistically strong solutions can be shown by a similar but simpler argument as in the following pathwise uniqueness proof. Now, due to the regularity of $u$ and $\chi$, we have $\mathbb P-a.s.$ that $u\otimes\chi\in L^2_{[0,T]}L^2$ and hence ${\rm div}(u\otimes\chi),~\partial_t\chi\in L^2_{[0,T]}H^{-1}$. Thus we obtain $\chi\in C_{[0,T]}L^2$ by \cite[Theorem 4.2.5]{LR15}.

We also introduce the following backward system on $[0,T]\times \mathbb{T}^d$ for the pair $(\Phi,\gamma)$:
\begin{equation}\label{backward linearized N-S system}
	\left\{\begin{aligned}
		&-\partial_t \Phi-u\cdot\nabla\Phi-\nu\Delta\Phi+\nabla\gamma=F~,\\
		&{\rm div}\ \Phi=0~,\\
		&\Phi(T)=0~.\\
	\end{aligned}\right.
\end{equation}
\begin{theorem}[{\cite[Theorem A.3]{CL22}}]\label{thm. regularity of linearized N-S}
Let $d\geq2$, $\nu>0$ and $0<T<\infty$ be arbitrarily fixed. Let $u\in X^{p,q}_T$ with some $ p\in[2,\infty]$ and $q\in (2,\infty]$ satisfying (\ref{Serrin scales}). Then for any given $F\in C^\infty_c\left([0,T]\times\mathbb T^d\right)$, the system (\ref{backward linearized N-S system}) has a weak solution $\Phi\in L^\infty_{[0,T]}H^1\bigcap L^2_{[0,T]}H^2$ with some $\nabla\gamma\in L^2_{[0,T]}L^2$, such that it can be used as a test function in the weak formulation
    \begin{align}
	\int_0^T\int_{\mathbb T^d}\eta\cdot\Big(\partial_t\Phi+u\cdot\nabla\Phi+\nu\Delta\Phi\Big)dxdt=0\nonumber
    \end{align}
    for $\eta\in L^\infty_{[0,T]}L^2\bigcap L^2_{[0,T]}H^1$.
\end{theorem}\

\par Now, with Theorem \ref{thm. existence of linearized sto. N-S} and \ref{thm. regularity of linearized N-S} in hand, we are ready to prove Theorem \ref{thm. Stochastic Serrin's Criterion}.

\begin{proof}[Proof for Theorem \ref{thm. Stochastic Serrin's Criterion}] Let $u$ be a solution to (\ref{sto. NS})$_\nu$ as stated in Theorem \ref{thm. Stochastic Serrin's Criterion}, with initial data $u_0$. By Theorem \ref{thm. existence of linearized sto. N-S}, there exists a probabilistically strong and analytically weak solution $\chi\in L^2\big(\Omega;~C_{[0,T]}L^2\big)\\\bigcap L^2\big(\Omega;~L^2_{[0,T]}H^1\big)$ to (\ref{linearized sto. N-S}) with initial data $\chi_0=u_0$ and satisfying (\ref{expectational energy inequality for linearized sto. N-S}). We first show that $\mathbb P-a.s.$ , $\chi=u$. Let $\eta:=u-\chi$
and we can write the equations for $\eta$ :
\begin{align*}
	\partial_t\eta+u \cdot\nabla\eta -\nu\Delta\eta +\nabla q =0~.
\end{align*}
This equation is understood in the sense of analytic weak formulation as follows:
\begin{align}\label{back}
	\int_0^T\int_{\mathbb T^d}\eta \cdot\Big(\partial_t\phi+u \cdot\nabla\phi+\Delta\phi\Big)dxdt=0
\end{align}
for all divergence-free test function $\phi\in C^\infty\left([0,T]\times\mathbb T^d\right)$ such that $\phi(T)=0$. Now for $u(\omega)\in X^{p,q}_T$ and arbitrarily given $F\in C^\infty_c\Big([0,T]\times\mathbb T^d\Big)$ applying Theorem \ref{thm. regularity of linearized N-S}, we obtain a solution $\Phi(\omega)$ to \eqref{backward linearized N-S system} with regularity stated in Theorem \ref{thm. regularity of linearized N-S}. Let $\phi=\Phi(\omega)$ in \eqref{back} we obtain
\begin{align}
	0
	&=-\int_0^T\int_{\mathbb T^d}\eta \cdot\Big(\partial_t\Phi +u \cdot\nabla\Phi +\Delta\Phi \Big)dxdt\nonumber\\
	&=\int_0^T\int_{\mathbb T^d}\eta (t,x)\cdot F(t,x)dxdt~,\nonumber
\end{align}
which {implies $\chi=u$ by continuity of $u$ and $\chi$}. (\ref{expectational energy inequality (normal)}) follows directly from (\ref{expectational energy inequality for linearized sto. N-S}).\
\par Next, we show  pathwise uniqueness. Let $u_1$, $u_2$ be two solutions to (\ref{sto. NS}) with the same initial data $u_0$, as stated in Theorem \ref{thm. Stochastic Serrin's Criterion}. Then by the first part of the proof, they are Leray-Hopf and belong to $C_{[0,T]}L^2\bigcap L^2_{[0,T]}H^1\bigcap X^{p,q}_T$, $\mathbb P-a.s.$. Now, let $\Upsilon:=u_1-u_2$.  Fix $\mathbb P-a.s.~\omega$, we have that $\Upsilon(\omega)\in C_{[0,T]}L^2\bigcap L^2_{[0,T]}H^1\bigcap X^{p,q}_T$
is a weak solution (in the sense of space-time distribution) to the Stokes system
\begin{equation*}
	\left\{\begin{aligned}
		&\partial_t \Upsilon =\Delta\Upsilon -\nabla p +{\rm div}F~,\\
		&{\rm div}\Upsilon \equiv0~,\\
		&\Upsilon(0)=0~,\\
	\end{aligned}\right.
\end{equation*}
with $F:=-\big(\Upsilon \otimes u_1 +u_2 \otimes\Upsilon \big)\in L^2_{[0,T]}L^2$.
Taking $L^2$-inner product on the both side of equation, we obtain
\begin{align}\label{energy equality of u_1-u_2}
	\frac{1}{2}\frac{{\rm d}}{{\rm d}t}\big\|\Upsilon (t)\big\|_{L^2}^2+\nu\big\|\nabla\Upsilon (t)\big\|_{L^2}^2
	&=\Big\langle{\rm div}F (t),\Upsilon (t)\Big\rangle\nonumber\\
	&=\Big\langle\Upsilon (t)\otimes u_1 (t),\nabla\Upsilon (t)\Big\rangle_{L^2}~,~~~t\in[0,T]~.
\end{align}
${\bf \bullet~Case 1:}$ $d<q\leq\infty$.
\par By H\"older's inequality with $\dfrac{1}{r}+\dfrac{1}{q}+\dfrac{1}{2}=1$, the Sobolev embedding $H^{d/q}\hookrightarrow L^r$, the interpolation and Young's inequality, we have for all $t\in[0,T]$ that
\begin{align}
	\left|\Big\langle\Upsilon (t)\otimes u_1 (t),\nabla\Upsilon (t)\Big\rangle_{L^2}\right|
	&\leq\big\|u_1 (t)\big\|_{L^q}\big\|\Upsilon (t)\big\|_{L^r}\big\|\nabla\Upsilon (t)\big\|_{L^2}\nonumber\\
	&\leq C\big\|u_1 (t)\big\|_{L^q}\big\|\Upsilon (t)\big\|_{H^{d/q}}\big\|\nabla\Upsilon (t)\big\|_{L^2}\nonumber\\
	&\leq C\big\|u_1 (t)\big\|_{L^q}\big\|\Upsilon (t)\big\|_{L^2}^{1-d/q}\big\|\nabla\Upsilon (t)\big\|_{L^2}^{1+d/q}\nonumber\\
	&\leq C\big\|u_1 (t)\big\|_{L^q}^{\frac2{1-d/p}}\big\|\Upsilon (t)\big\|_{L^2}^2+\frac{1}{2}\big\|\nabla\Upsilon (t)\big\|_{L^2}^2~.\nonumber
\end{align}
Using (\ref{Serrin scales}), we have for all $t\in[0,T]$ that
\begin{align}
	\frac{{\rm d}}{{\rm d}t}\big\|\Upsilon (t)\big\|_{L^2}^2+\big\|\nabla\Upsilon (t)\big\|_{L^2}^2\leq C\Big(\big\|u_1 (t)\big\|_{L^q}^p+1\Big)\big\|\Upsilon (t)\big\|_{L^2}^2~.\nonumber
\end{align}
Here, $C>0$ is a deterministic and finite constant. {Applying Gronwall's inequality, then by the facts that $\big\|u_1 \big\|_{L^p_{[0,T]}L^q}<\infty$ and $\Upsilon(0)=0$ for $\mathbb P-a.s.~\omega$ , we obtain the desired uniqueness.}\\\\
${\bf \bullet~Case 2:}$ $d=q<\infty$.
\par Since $u_i \in C_{[0,T]}L^d$, we {have for any $\varepsilon>0$ } there exist $\overline{u}_1,u_\varepsilon$ such that $u_1 =\overline{u}_1 +u_\varepsilon $ with $\big\|\overline{u}_1 \big\|_{C_{[0,T]}L^d}<\varepsilon$ and $\big\|u_\varepsilon \big\|_{L^\infty([0,T]\times\mathbb T^d)}<\infty$.
\par Then for the $\overline{u}_1 $ part, by H\"older's inequality with $\dfrac{1}{r}+\dfrac{1}{d}+\dfrac{1}{2}=1$, the Sobolev imbedding $H^{1}\hookrightarrow L^r$, we have for all $t\in[0,T]$ that
\begin{align}
	\left|\Big\langle\Upsilon (t)\otimes\overline{u}_1 (t),\nabla\Upsilon (t)\Big\rangle_{L^2}\right|
	&\leq\big\|\overline{u}_1 (t)\big\|_{L^d}\big\|\Upsilon (t)\big\|_{L^r}\big\|\nabla\Upsilon (t)\big\|_{L^2}\nonumber\\
	&\leq C\varepsilon\big\|\nabla\Upsilon (t)\big\|_{L^2}^2~,\nonumber
\end{align}
and for the $u_\varepsilon $ part, by H\"older's inequality and Young's inequality, we have for all $t\in[0,T]$ that
\begin{align}
	\left|\Big\langle\Upsilon (t)\otimes u_\varepsilon (t),\nabla\Upsilon (t)\Big\rangle_{L^2}\right|
	&\leq\big\|u_\varepsilon \big\|_{L^\infty}\big\|\Upsilon (t)\big\|_{L^2}\big\|\nabla\Upsilon (t)\big\|_{L^2}\nonumber\\
	&\leq4\big\|u_\varepsilon \big\|_{L^\infty}^2\big\|\Upsilon (t)\big\|_{L^2}^2+\frac{1}{4}\big\|\nabla\Upsilon (t)\big\|_{L^2}^2~.\nonumber
\end{align}
Here, $C>0$ is a deterministic and finite constant that only depends on $d$. Then by choosing $\varepsilon=\frac{1}{4 C}$, and together with (\ref{energy equality of u_1-u_2}), we have for all $t\in[0,T]$ that
\begin{align}
	\frac{{\rm d}}{{\rm d}t}\big\|\Upsilon (t)\big\|_{L^2}^2+\big\|\nabla\Upsilon (t)\big\|_{L^2}^2\leq8\big\|u_\varepsilon \big\|_{L^\infty}^2\big\|\Upsilon (t)\big\|_{L^2}^2~.\nonumber
\end{align}
Then using Gronwall's inequality together with the fact that $\Upsilon(0)=0$ implies the desired uniqueness.

\end{proof}\

\appendix
  \renewcommand{\appendixname}{Appendix~\Alph{section}}
   \renewcommand{\theequation}{A.\arabic{equation}}
   \section{Stationary Mikado Flows and its div-Potential}
\label{Appendix A}\
\par In this part we recall the construction of stationary Mikado flows from \cite{CL22}. We point out that the construction is entirely deterministic. Let us begin with the following geometric lemma (cf. \cite[Lemma 1]{Nas54} \cite[Lemma 2.4]{DS17}). Recall that $S_+^{d\times d}$ is the set of all positive definite and symmetric $d\times d$ matrices.

\begin{lemma}[\bf Geometric Lemma]\label{lem. Geometry}
For any compact subset $\mathcal K\subset S_+^{d\times d}$, there exist a finite set $\Lambda\subset\mathbb Z^d$ and smooth functions $\Gamma_k\in C^\infty\left(\mathcal K;\mathbb R\right)$ such that for all $R\in\mathcal K$,
\begin{equation} \label{geometric decomposition for matrices}
	R=\sum_{k\in\Lambda}\Gamma_k^2\left(R\right)\mathbf{e}_k\otimes\mathbf{e}_k~,
\end{equation}
where $\mathbf{e}_k:=k/|k|$ for each $k\in\Lambda$~.	
\end{lemma}
In the following we always choose $\mathcal K=B_{1/2}\left(\textrm{Id}\right)$.
\par {The construction of stationary Mikado flows is as follows.  We first choose a point $p_k\in(0,1)^d$ {for each $k\in\Lambda$} such that $p_k\neq p_{-k}$ if both $k,-k\in\Lambda$. For each $k\in\Lambda$ we denote a periodic line $l_k:=\lbrace sk+p_k\in\mathbb T^d:s\geq0\rbrace$ that passes through $p_k$ in direction $\mathbf{e}_k$.

Now let  $\Psi,\Phi\in C_c^\infty\left((1/2,1);\mathbb R \right)$  {and $c_k>0$} be set such that the functions
\begin{align}
	&\Psi_k(x):={\mu}^{\frac{d-1}{2}}{c_k}\Psi\big({\mu}{\rm dist}(l_k,x)\big)~,~~~~x\in\mathbb T^d;\nonumber\\
	&\Phi_k(x):={\mu}^{\frac{d-1}{2}-2} {c_k}\Phi\big({\mu}{\rm dist}(l_k,x)\big)~,~~~~x\in\mathbb T^d\nonumber
\end{align}
satisfying
\begin{align} \label{normalized L^2-norm of Psi_k}
	&\big\| {\Psi_k}\big\|_{L^2}=1~,\\
	&\Delta {\Phi_k}= {\Psi_k}~~~\textrm{on}~\mathbb T^d~.\nonumber
\end{align}
{Here, the constant $\mu>0$ is the so-called concentration parameter and always set to be sufficiently large for special use. Then the stationary Mikado flows and {their} div-potential are defined by
\begin{align}\label{definition for W_k}
	 {\mathbb W_k}
	&:={\Psi_k}\mathbf{e}_k~,\\
	\label{definition for V_k}
	 {\mathbb V_k}
	&:=\mathbf{e}_k\otimes\nabla {\Phi_k}-\nabla {\Phi_k}\otimes\mathbf{e}_k~.
\end{align}
For convenience, we recall the following equality and bounds from \cite[Theorem 4.3]{CL22}.
\begin{align}
	&{\rm div}~\mathbb W_k=0~,\nonumber\\
	\label{stationary Euler equation for W_k}
	&{\rm div}\left(\mathbb W_k\otimes\mathbb W_k\right) =0~,\\
	\label{relation divV_k=W_k}
	&\textrm{div}\mathbb V_k=\mathbb W_k~,
\end{align}
and
\begin{align}\label{estimates for Mikado flows W0W}
	&\left\|\mathbb W_k\otimes\mathbb W_{k'}\right\|_{L^\alpha}\lesssim \mu^{(d-1)-\frac{d}{\alpha}}~~~,~for~all~k\neq k'~;\\
	&\mu^{-m}\left\|\nabla^m\mathbb W_k\right\|_{L^\alpha}\lesssim_m\mu^{\frac{d-1}{2}-\frac{d-1}{\alpha}}~~~,~for~all~k~;\label{estimates for Mikado flows W}\\
	&\mu^{-m}\left\|\nabla^m\mathbb V_k\right\|_{L^\alpha}\lesssim_m\mu^{-1+\frac{d-1}{2}-\frac{d-1}{\alpha}}~~~,~for~all~k~;\label{estimates for Mikado flows V}
\end{align}
for all $m\in\mathbb N_0$ and all $1\leq \alpha\leq\infty$.\\

\renewcommand{\theequation}{B.\arabic{equation}}
\section{The Operators $\mathcal R$ and $\mathcal B$}
\label{Appendix B}
In this section we recall anti-divergence operator $\mathcal{R}$ and $\mathcal{B}$ from \cite[Appendix B]{CL22}.\\

\noindent$\bullet$ {\bf Anti-divergence $\mathcal R$}

The operator $\mathcal R: C^\infty\big(\mathbb T^d;\mathbb R^d\big)\longrightarrow C^\infty\big(\mathbb T^d;\mathcal S^{d\times d}_0\big)$
\begin{align}
	(\mathcal R v)^{ij}:=\mathcal R^{ij}_kv^k:=\frac{2-d}{d-1}\Delta^{-2}\partial_i\partial_j\partial_kv^k-\frac{\delta_{ij}}{d-1}\Delta^{-1}\partial_kv^k+\Delta^{-1}\partial_i\delta_{jk}v^k+\Delta^{-1}\partial_j\delta_{ik}v^k~,\nonumber
\end{align}
Here, $\mathcal S^{d\times d}_0$ denotes the set of $d\times d$ trace-free symmetric matrices. A direct computation (see \cite[Appendix B.2]{CL22}) gives that
\begin{align}
	{\rm tr}(\mathcal R v)&=0~,\nonumber\\
	{\rm div}(\mathcal R v)&=v-\fint_{\mathbb T^d}v~,\nonumber\\
	 {\mathcal{R}\Delta v}& {=\nabla v+\nabla^{\rm T}v~.\label{relation R-Laplace}}
\end{align}
It can be shown that $\mathcal R$ is $L^p$-bounded for $1\leq p\leq\infty$ ( see \cite[Theorem B.3]{CL22})  and 
\begin{align}\label{peri-L^p estimate of anti-divergence R}
	\big\|\mathcal{R} f(\sigma\cdot)\big\|_{L^p}\lesssim\sigma^{-1}\big\|f\big\|_{L^p}~
\end{align}
for $f$ with mean zero.
\\\\
\noindent$\bullet$ {\bf  Bilinear anti-divergence $\mathcal B$}

The operator $\mathcal B: C^\infty\big(\mathbb T^d;\mathbb R^d\big)\times C^\infty\big(\mathbb T^d;\mathbb R^{d\times d}\big)\longrightarrow C^\infty\big(\mathbb T^d;\mathcal S^{d\times d}_0\big)$ is defined by
\begin{align}
	\big(\mathcal B(v,M)\big)_{ij}:=v^l\mathcal R^{ij}_kM^k_l+\mathcal R(\partial_iv^l\mathcal R^{ij}_kM^k_l)~.\nonumber
\end{align}
Let $C^\infty_0\big(\mathbb T^d;\mathbb R^{d\times d}\big)$ be the set of periodic smooth matrix-valued functions with zero mean. Then for $v\in C^\infty\big(\mathbb T^d;\mathbb R^d\big)$ and $M\in C^\infty_0\big(\mathbb T^d;\mathbb R^{d\times d}\big)$, we have (see \cite[Theorem B.4]{CL22})
\begin{align}
	{\rm div}\big(\mathcal B(v,M)\big)
	&=vM-\fint_{\mathbb T^d}vM~,\nonumber\\
	\big\|\big(\mathcal B(v,M)\big)\big\|_{L^p}
	&\lesssim\big\|v\big\|_1\big\|M\big\|_{L^p}~,~~~\forall 1\leq p\leq\infty~.\label{L^p estimate for bilinear reverse divergence}
\end{align}\

\renewcommand{\theequation}{C.\arabic{equation}}
\section{$C^N$-Estimates for Compositions}
\label{Appendix C}\
\par We recall the following lemma from \cite[Proposition C.1]{BLIL15}.
\begin{lemma}
	Let $\Psi:\Omega\longrightarrow\mathbb R$ and $f:\mathbb R^n\longrightarrow\Omega$ be two smooth functions, with $\Omega\subset\mathbb R^m$. Then, for every $N\in\mathbb N$ , there is a constant $C=C(n,m,N)>0$ such that
	\begin{align}\label{C^N-estimate for compositions 1}
		\big[\Psi\circ f\big]_N
		&\leq C\left(\big[\Psi\big]_1\big[ f\big]_N+\big\|{\rm D}\Psi\big\|_{N-1}\big\|f\big\|_0^{N-1}\big[ f\big]_N\right),\\
		\big[\Psi\circ f\big]_N
		&\leq C\left(\big[\Psi\big]_1\big[ f\big]_N+\big\|{\rm D}\Psi\big\|_{N-1}\big[ f\big]_1^N\right).\label{C^N-estimates for compositions 2}
	\end{align}
\end{lemma}\
\bibliographystyle{alpha}
\bibliography{Paper1ref}

\end{document}